\documentclass{amsart}

\usepackage[mathscr]{eucal}% So-called Euler fonts, allows \mathscr
\usepackage{amssymb}% allows special arrows like >-> e.g.
\usepackage{pifont}
\usepackage{graphicx}
\usepackage{psfrag} % Necessary to put latex into .eps text fragments
\usepackage[usenames,dvipsnames]{color}
\usepackage[normalem]{ulem}% allows \sout to cross-out text.
\usepackage{amsthm}% allows theorem* , e.g.
\usepackage{bbm}% same as above but slightly less nice; use if \mathbb should remain as `normal'
\usepackage{enumerate}% allows easy change of label
\usepackage{array}% allows array
\usepackage{amsmath}% allows pmatrix

\usepackage{hyperref}
\hypersetup{colorlinks=true,linkcolor={Brown},citecolor={Brown},urlcolor={Brown}}% replace BrickRed by Brown?

\usepackage{cleveref}% allows references via \cref and \Cref which automatically repeat the name (Thm, Prop, etc) of the ref.

\numberwithin{equation}{section}% makes equat numb contain the section % superseded by the above
\makeindex

\usepackage[all]{xy}
\CompileMatrices

\newdir{ >}{{}*!/-10pt/\dir{>}}

\swapnumbers %pour que le numéro apparaisse devant le théorème

\newtheorem{Thm}[equation]{Theorem}
\newtheorem{Prop}[equation]{Proposition}
\newtheorem{Lem}[equation]{Lemma}
\newtheorem{Cor}[equation]{Corollary}

\theoremstyle{remark}
\newtheorem{Rem}[equation]{Remark}
\newtheorem{Def}[equation]{Definition}
\newtheorem{Ter}[equation]{Terminology}
\newtheorem{Not}[equation]{Notation}
\newtheorem{Exa}[equation]{Example}
\newtheorem{Exas}[equation]{Examples}
\newtheorem{Cons}[equation]{Construction}
\newtheorem{Conv}[equation]{Convention}
\newtheorem{Hyp}[equation]{Hypotheses}
\newtheorem{Rec}[equation]{Recollection}

\newtheorem{Warning}[equation]{Warning}

\theoremstyle{definition}
\newtheorem*{Ack}{Acknowledgements}

\newcommand{\nc}{\newcommand}
\nc{\dmo}{\DeclareMathOperator}

\dmo{\Ab}{Ab}
\dmo{\add}{add} % the additive hull (= closure under sums and direct summands)
\dmo{\Aut}{Aut}
\dmo{\bicMack}{\biMack^{\mathsf{ic}}} % the bicategory of i.c. Mackey 2-functors
\dmo{\bicMackk}{\biMack^{\mathsf{ic}}_{\kk}} % ... k-linear version
\dmo{\bicCoMackk}{\biCoMack^{\mathsf{ic}}_{\kk}} % ... cohomological version
\dmo{\biMack}{\mathsf{Mack}} % the bicategory of Mackey 2-functors
\dmo{\biCoMack}{\mathsf{CohMack}} % ... the cohomological version
\dmo{\biMackk}{\mathsf{Mack}_{\kk}} % ... k-linear version
\dmo{\biCoMackk}{\mathsf{CohMack}_{\kk}} % ... k-linear version
\dmo{\Bimod}{\mathsf{Bimod}}
\dmo{\biker}{\mathsf{ker}}
\dmo{\Blah}{Blah}
\dmo{\Ch}{Ch}% ground notation for chain complexes
\dmo{\CoInd}{CoInd}
\dmo{\CoMack}{CohMack}
\dmo{\Der}{D}% ground notation for derived categories
\dmo{\End}{End}
\dmo{\Fun}{\mathrm{Fun}} % the 2-category of 2-functors from enriched category theory
\dmo{\Hom}{Hom}
\dmo{\Ho}{Ho}
\dmo{\img}{im}
\dmo{\incl}{incl}
\dmo{\Ind}{Ind}
\dmo{\ind}{ind}
\dmo{\Inj}{Inj} % injective modules/objects
\dmo{\Ker}{Ker}
\dmo{\Mackey}{Mack} % the category of ordinary Mackey functors
\dmo{\Map}{Map}%
\dmo{\Mod}{\mathsf{Mod}}% sheaves of modules
\dmo{\modcat}{mod} % category of f.g. (left) modules
\dmo{\Modcat}{Mod} % category of (left) modules
\dmo{\Mor}{Mor}%
\dmo{\mot}{\mathsf{mot}} % the Mackey 2-motive associated to a groupoid
\dmo{\cohmot}{\mathsf{mot}^{\mathsf{coh}}} % the coh Mackey 2-motive associated to a groupoid
\dmo{\Mot}{\mathsf{Mot}} % Mackey 2-motives
\dmo{\Motk}{\Mot_{\kk}} % ... k-linear version
\dmo{\CohMotk}{\mathsf{Mot}^{\mathsf{coh}}_{\kk}} %
\dmo{\Obj}{Obj}
\dmo{\perm}{\mathsf{biperm}} % bicategory of permutation bimodules
\dmo{\biperm}{\mathsf{biperm}} % bicategory of permutation bimodules
\dmo{\permcat}{perm}
\dmo{\dashperm}{\text{-}\mathsf{perm}}%
\dmo{\dashpperm}{\text{-}\mathsf{pperm}}%
\dmo{\Proj}{Proj} % projective modules/objects
\dmo{\pr}{pr}
\dmo{\PsFunJJ}{\PsFun_{\JJ_!}^{\JJ^\prime\textrm{\!-}\mathsf{oplax}}}
\dmo{\PsFunJop}{\PsFun_{{{\JJ}_{{}_{*}}}}}
\dmo{\PsFunJ}{\PsFun_{\JJ_!}}
\dmo{\PsFunoplax}{\PsFun^{\mathsf{oplax}}}
\dmo{\PsFun}{\mathsf{PsFun}} % the bicategory of pseudo-functors
\dmo{\Res}{Res}
\dmo{\SH}{SH}% ground name for cat of spectra
\dmo{\Spancat}{Sp}
\dmo{\Spanname}{{\sf Span}}
\dmo{\Stab}{Stab}% stable category of non-fin. gen. mod.
\dmo{\twoFun}{2\mathsf{Fun}}

\nc\noloc{\nobreak\mspace{6mu plus 1mu}{:}\nonscript\mkern-\thinmuskip\mathpunct{}\mspace{2mu}}% mirror of \colon
\nc{\ababs}{{\sl ab absurdo}}
\nc{\Add}{\mathsf{Add}}
\nc{\ADD}{\mathsf{ADD}}
\nc{\ADDick}{\ADD^{\mathsf{ic}}_{\kk}}
\nc{\adhoc}{{\sl ad hoc}}
\nc{\adjto}{\rightleftarrows}
\nc{\adj}{\dashv\,}
\nc{\afortiori}{{\sl a fortiori}}
\nc{\aka}{{a.\,k.\,a.}\ }
\nc{\all}{\mathsf{all}}% all 1-cells.
\nc{\apriori}{{\sl a priori}}
\nc{\ass}{\mathrm{ass}} % associator
\nc{\bbA}{\mathbb{A}}
\nc{\bbB}{\mathbb{B}}
\nc{\bbC}{\mathbb{C}}
\nc{\bbD}{\mathbb{D}}
\nc{\bbF}{\mathbb{F}}
\nc{\bbI}{\mathbb{I}}
\nc{\bbM}{\mathbb{M}}
\nc{\bbN}{\mathbb{N}}
\nc{\bbP}{\mathbb{P}}
\nc{\bbQ}{\mathbb{Q}}
\nc{\bbR}{\mathbb{R}}
\nc{\bbZ}{\mathbb{Z}}
\nc{\bs}{\backslash}
\nc{\BurnG}{\cat{A}(G)}
\nc{\cat}[1]{\mathcal{#1}}
\nc{\Cat}{\mathsf{Cat}}
\nc{\CAT}{\mathsf{CAT}}
\nc{\cf}{{\sl cf.}\ }
\nc{\Cf}{{\sl Cf.}\ }
\nc{\coeq}{\mathop{\mathrm{coeq}}}
\nc{\colim}{\mathop{\mathrm{colim}}}
\nc{\costar}{**}% for the (-)^* embedding with \beta^{-1} on 2-cells
\nc{\co}{{\mathrm{co}}}
\nc{\DD}{\cat{D}}% a derivator
\nc{\Displ}{\displaystyle}
\nc{\doublequot}[3]{#1\backslash #2/#3}% double-cosets
\nc{\Ecell}{\rotatebox[origin=c]{90}{$\Downarrow$}} % treated as others too! But not to be used in-line!!! [I]
\nc{\eg}{{\sl e.g.}\ } % made similar to the previous two
\nc{\Eg}{{\sl E.g.}\ } % we needed a capital one!
\nc{\eps}{\varepsilon}
\nc{\equalby}[1]{\overset{\textrm{#1}}{=}}
\nc{\exact}{\mathsf{ex}}
\nc{\faithful}{\mathsf{faithful}}% faithful
\nc{\faith}{\mathsf{faithf}}
\nc{\final}{\textrm{\scriptsize{\ding{93}}}} % modified asterisk
\nc{\Funadd}{\Fun_{\amalg}}% additive functors, in the sense of Mack 1
\nc{\Funplus}{\Fun_{+}}% additive functors. Can be tweaked...
\nc{\fun}{\mathrm{fun}} % "funtorator" of pseudo-functors (or whatever it's called)
\nc{\GG}{\mathbb{G}}% the 2-category of finite groupoids `of interest'
\nc{\gpdG}{{\groupoidf_{\!\smallslash\!G}}} % the "correct" 2-category of groupoids for Mackey functors for G, i.e. the comma category of faithful functors to G
\nc{\gpdHG}{{\groupoidf_{\!\smallslash\!H\times G}}} %
\nc{\gpdGG}{{\groupoidf_{\!\smallslash\!G_1\times G_2}}} %
\nc{\gpd}{\groupoid}% 2-category of finite groupoids
\nc{\gps}{\mathsf{groups}} % 2-category of finite groups
\nc{\groconn}{\groupoid_{\mathsf{conn}}}% connected finite groupoids (auxiliary)
\nc{\groupoidf}{\groupoid{}^{\smallfaithful}}% finite groupoids with faithful morphisms
\nc{\gpdf}{\groupoidf}%
\nc{\groupoid}{\mathsf{gpd}}% finite groupoids
\nc{\group}{\mathsf{group}} % category of finite groups, variant...
\nc{\Gsets}{G\sset}
\nc{\HGfK}{\doublequot{H}{G}{f(K)}}%
\nc{\HGK}{\doublequot HGK}% most used
\nc{\Homcat}[1]{\Hom_{\cat #1}}
\nc{\hooklongleftarrow}{\longleftarrow\joinrel\rhook}
\nc{\hooklongrightarrow}{\lhook\joinrel\longrightarrow}
\nc{\hook}{\hookrightarrow}
\nc{\Hsets}{H\mathsf{-sets}}
\nc{\ICAdd}{\Add_{\mathsf{ic}}}%
\nc{\ICADD}{\ADD_{\mathsf{ic}}}%
\nc{\ICADDkk}{\ADD_{\mathsf{ic},\kk}}%
\nc{\Idcat}[1]{\Id_{\cat{#1}}}
\nc{\id}{\mathrm{id}}
\nc{\Id}{\mathrm{Id}}
\nc{\ie}{{\sl i.e.}\ }
\nc{\into}{\mathop{\rightarrowtail}}
\nc{\inv}{^{-1}}
\nc{\Ivout}[1]{\Ivo{\sout{#1}}}
\nc{\isocell}[1]{\undersett{ #1}{\overset{\sim}{\Ecell}}} % to be used ONLY for 2-cells in xypic diagrams [I]
\nc{\Isocell}[1]{\undersett{ #1}{\overset{\sim}{\Longrightarrow}}}% ONLY inline with target and source [I]
\nc{\isoEcell}{\overset{\sim}{\Rightarrow}} % to be used ONLY in-line, with target and source [I]
\nc{\isotoo}{\stackrel{\sim}\longrightarrow}
\nc{\isoto}{\buildrel \sim\over\to}
\nc{\Ivo}[1]{{\color{OliveGreen}#1}}
\nc{\JJ}{\mathbb{J}}% old \II % the class of `faithful' guys, we can try other things.
\nc{\kk}{\Bbbk}
\nc{\KK}{\mathrm{KK}}
\nc{\leps}{{}^{\ell}\eps}
\nc{\leta}{{}^{\ell}\eta}
\nc{\loccit}{{\sl loc.\ cit.}}
\nc{\lotoo}[1]{\overset{#1}{\,\longleftarrow\,}}
\nc{\loto}[1]{\overset{#1}{\leftarrow}}
\nc{\lto}{\leftarrow}
\nc{\lun}{\mathrm{lun}} % left unitor
\nc{\Mackintro}[1]{(Mack\,\ref{Mack-#1-intro})}
\nc{\Mack}[1]{(Mack\,\ref{Mack-#1})}
\nc{\Mid}{\,\big|\,}
\nc{\MMod}{\,\text{-}\Mod}%
\nc{\mmod}{\,\text{-mod}}%
\nc{\ffree}{\,\text{-free}}
\nc{\MM}{\cat{M}}% a Mackey 2-functor
\nc{\Muniv}{\cat{M}_{\mathsf{univ}}}
\nc{\Ncell}{\rotatebox[origin=c]{0}{$\Uparrow$}} % treated as the others, for uniform size
\nc{\NEcell}{\rotatebox[origin=c]{135}{$\Downarrow$}} % North-East oriented 2-cell arrow
\nc{\NN}{\cat{N}}% another Mackey 2-functor
\nc{\NWcell}{\rotatebox[origin=c]{-135}{$\Downarrow$}} % North-West oriented 2-cell arrow
\nc{\oEcell}[1]{\overset{\scriptstyle #1}{\Ecell}} % to be used ONLY for 2-cells in xypic diagrams [I]
\nc{\oWcell}[1]{\overset{\scriptstyle #1}{\Wcell}} % same, but for Wcell
\nc{\ointo}[1]{\overset{#1}{\rightarrowtail}}
\nc{\olto}[1]{\overset{#1}\lto}
\nc{\onto}{\mathop{\twoheadrightarrow}}
\nc{\op}{{\mathrm{op}}}
\nc{\otoo}[1]{\overset{#1}{\,\longrightarrow\,}}
\nc{\oto}[1]{\overset{#1}\to}
\nc{\Paul}[1]{{\color{Orange}#1}}
\nc{\pih}[1]{\tau_{1}#1}%
\nc{\Pout}[1]{\Paul{\sout{#1}}}
\nc{\PsFunJindex}{\PsFun_{{\JJ_!}} \ \ {{\JJ}_{!}}\textrm{-strong pseudo-functors}}% hideous trick to circumvent the index problem with \PsFunJ
\nc{\qquadtext}[1]{\qquad\textrm{#1}\qquad}
\nc{\quadtext}[1]{\quad\textrm{#1}\quad}
\nc{\ra}{\rightarrow}
\nc{\reps}{{}^{r\!}\eps}
\nc{\restr}[1]{{|_{\scriptstyle #1}}}% to restrict a map
\nc{\reta}{{}^{r\!}\eta}
\nc{\run}{\mathrm{run}} % right unitor
\nc{\Sad}{\mathsf{Sad}}
\nc{\SAD}{\mathsf{SAD}}
\nc{\sbull}{{\scriptscriptstyle\bullet}}
\nc{\Scell}{\rotatebox[origin=c]{0}{$\Downarrow$}} % treated as the others, for uniform size
\nc{\SEcell}{\rotatebox[origin=c]{45}{$\Downarrow$}} % South-East oriented 2-cell arrow
\nc{\SET}[2]{\big\{\,#1\Mid#2\,\big\}}
\nc{\set}{\mathrm{set}} % category of finite sets%
\nc{\sethat}{\widehat{\mathsf{set}}} % category of finite sets
\nc{\Set}{Set} % non-sf version...
\nc{\smallfaithful}{\mathsf{f}}% faithful
\nc{\smallslash}{{}^{\scriptscriptstyle/}}
\nc{\smat}[1]{\left(\begin{smallmatrix} #1 \end{smallmatrix}\right)}
\nc{\spanG}{{\widehat{\mathsf{gp}\,\,}\!\!\mathsf{d}}{}^\smallfaithful_{\!{}^{\scriptscriptstyle/}\!G}}% span of groupoids over G.
\nc{\Spanhat}{\textrm{\sf S}\widehat{\textrm{\sf pan}}} %
\nc{\Spanhatrf}{\Spanhat{}^\rfree} % short for semi-additive motives...
\nc{\Span}{\Spanname}% the span bicategory of a (2,1)-category
\nc{\spancat}{\mathrm{Sp}} % ordinary category of spans of a (2,1)-category pair
\nc{\spank}{\spancat_{\kk}}
\nc{\biset}{\mathsf{biset}} % the biset bicategory
\nc{\bbiset}{\textrm{-}\biset} %
\nc{\bbisetrf}{\textrm{-}\biset^\rfree} %
\nc{\bisethatrf}{\mathsf{bi}\widehat{\mathsf{set}}{}^\rfree} % the right-free biset bicategory + vertical spans
\nc{\bisethatHfree}{\mathsf{b}\widehat{\mathsf{iset}}{}^{H\textsf{-free}}} % the image under transposition of the above !
\nc{\bisetcat}{\mathrm{Bis}} % the ordinary category of bisets
\nc{\bisetcatbif}{\mathrm{Bis}^\mathrm{bif}_\kk\!} % bifree version
\nc{\bisetcatrf}{\mathrm{Bis}^\mathrm{rf}_\kk\!} % rightfree version
\nc{\bisetfun}{\mathrm{BisF}} % the category of biset functors
\nc{\bifree}{\mathsf{bif}}
\nc{\rfree}{\mathsf{rf}}
\nc{\lfree}{\mathsf{lf}}
\nc{\sset}{\textrm{-}\set}
\nc{\ssethat}{\textrm{-}\sethat}
\nc{\str}{\mathsf{str}}
\nc{\SWcell}{\rotatebox[origin=c]{-45}{$\Downarrow$}} % South-West oriented 2-cell arrow
\nc{\too}{\mathop{\longrightarrow}\limits}
\nc{\tristars}{\begin{center} $ *\ *\ * $ \end{center}}
\nc{\tSpan}{\pih{\Spanname}}% 1-truncation of Span
\nc{\undersett}[1]{\underset{\scriptstyle #1}}
\nc{\un}{\mathrm{un}} % "unitor" of pseudo-functors
\nc{\vcorrect}[1]{{\vphantom{\vbox to #1em{}}}}
\nc{\Wcell}{\rotatebox[origin=c]{90}{$\Uparrow$}} % treated as the others, for uniform size
\nc{\what}[1]{\widehat{\cat{#1}}}% span category with name #1.
\nc{\xto}{\xrightarrow}
\nc{\xlto}{\xleftarrow}

\nc{\xBur}{\mathrm{B^c}} % crossed Burnside ring
\nc{\xBurk}{ \mathrm{B}^{\mathrm{c}}_{\kk} } % crossed Burnside k-algebra
\nc{\Bur}{\mathrm{B}} % ordinary Burnside ring
\nc{\Burk}{\Bur_{\kk}} % ordinary Burnside k-algebra

\nc{\To}{\Rightarrow}
\nc{\isoTo}{\buildrel\sim\over\To}
\nc{\Too}{\Longrightarrow}

\nc{\bipermk}{\perm^\rfree_\kk}
\nc{\bP}{\overline{\cat{P}}} %

\begin{document}

%------------------------------------------------------------------------------

\title{Cohomological Mackey 2-functors}
\author{Paul Balmer}
\author{Ivo Dell'Ambrogio}
\date{\today}

\address{\ \vfill
\noindent PB: UCLA Mathematics Department, Los Angeles, CA 90095-1555, USA}
\email{balmer@math.ucla.edu}
\urladdr{http://www.math.ucla.edu/$\sim$balmer}

\address{\ \medbreak
\noindent ID: Univ.\ Lille, CNRS, UMR 8524 - Laboratoire Paul Painlev\'e, F-59000 Lille, France % nouvel adresse selon la Charte de signature de l'UniversitŽ !
}
\email{ivo.dell-ambrogio@univ-lille.fr}
\urladdr{http://math.univ-lille1.fr/$\sim$dellambr}

\begin{abstract} \normalsize
We show that the bicategory of finite groupoids and right-free permutation bimodules is a quotient of the bicategory of Mackey 2-motives introduced in \cite{BalmerDellAmbrogio20}, obtained by modding out the so-called cohomological relations. This categorifies Yoshida's Theorem for ordinary cohomological Mackey functors, and provides a direct connection between Mackey 2-motives and the usual blocks of representation theory.
\end{abstract}

\thanks{First-named author partially supported by NSF grant~DMS-1901696.}
\thanks{Second-named author partially supported by Project ANR ChroK (ANR-16-CE40-0003) and Labex CEMPI (ANR-11-LABX-0007-01).}


\maketitle

%------------------------------------------------------------------------------

\tableofcontents

%------------------------------------------------------------------------------
%------------------------------------------------------------------------------
%
\section{Introduction}
\label{sec:introduction}%
%\bigbreak
%------------------------------------------------------------------------------

We consider Mackey 2-functors in the sense of~\cite{BalmerDellAmbrogio20} and develop three themes.
The first two are general. They consist of a decategorification result and a new simpler description of the Mackey 2-motives of~\cite{BalmerDellAmbrogio20}. The third theme is the more specific purpose of this paper, the study of \emph{cohomological} Mackey 2-functors.
Let us now explain in some detail the interplay of these three topics.

\tristars

A \emph{Mackey 2-functor}~$\MM$ consists of additive categories $\MM(G)$ depending on finite groupoids~$G$ and whose variance in~$G$ is reminiscent of ordinary Mackey functors~\cite{Webb00}. This categorification, recalled in \Cref{sec:reminder}, involves restriction functors $u^*\colon \MM(G)\to \MM(K)$ for every morphism $u\colon K\to G$ and \emph{two-sided} adjoints $i_!= i_*\colon \MM(H)\to \MM(G)$ to $i^*$ for every faithful $i\colon H\into G$. The commutation of $i_!$ and $u^*$ is governed by a \emph{2-Mackey formula}~\eqref{eq:Mackey-square}. As discussed at length in~\cite[Ch.\,4]{BalmerDellAmbrogio20}, this type of structure has been used for a long time in a variety of settings.
For instance, in representation theory, the category $\MM(G)=\Modcat (\kk G)$ of $\kk G$-modules is an example of a Mackey 2-functor. In topology, the equivariant stable homotopy category $\MM(G)=\SH(G)$ of genuine $G$-spectra is another one.

Recall~\cite[\S\,7]{Webb00} that an ordinary Mackey (1-)\,functor~$M$ is \emph{cohomological} if, for every subgroup~$H\le G$, the composite $I_H^G\,R^G_H\colon M(G)\to M(G)$ of the restriction homomorphism $R^G_H\colon M(G)\to M(H)$ with the induction (transfer) homomorphism $I_H^G\colon M(H)\to M(G)$ is equal to multiplication by the index~$[G\!:\!H]$ on the abelian group~$M(G)$.
It might be tempting to define a cohomological Mackey 2-functor~$\MM$ as one for which the composite of~$i^*\colon \MM(G)\to \MM(H)$ with $i_!\colon \MM(H)\to \MM(G)$ is some form of multiplication by~$[G\!:\!H]$ on~$\MM(G)$.
This is not a good idea, however, if only because of a lack of interesting examples.
We propose here a definition that shifts the cohomological relation to the level of 2-cells, as follows:

\begin{Def}
\label{Def:co2mack-intro}
A Mackey 2-functor $\MM$ is \emph{cohomological} if, whenever $i\colon H\into G$ is the inclusion of a subgroup $H$ in a finite group~$G$, the composite
\[
\xymatrix{
\Id_{\MM(G)} \ar@{=>}[r]^-{\reta} & \Ind^G_H\Res^G_H \ar@{=>}[r]^-{\leps} & \Id_{\MM(G)}
}
\]
equals multiplication by the index $[G\!:\!H]$, where $\reta$ is the unit of the right adjunction $\Res^G_H=i^* \dashv i_*=\Ind_H^G$ and $\leps$ the counit of the left one  $\Ind_H^G=i_! \dashv i^*=\Res_H^G$.
\end{Def}

The Mackey 2-functors arising in representation theory are often cohomological, like the above $\MM(G)=\Modcat (\kk G)$ or the derived version $\MM(G)=\Der(\Modcat (\kk G))$. We present further examples of a more geometric nature in \Cref{sec:cohom-2-Mackey}.

The reader should be warned that the standard decategorification of a Mackey 2-functor via the Grothendieck group~$K_0$, as in~\cite[\S\,2.5]{BalmerDellAmbrogio20}, does \emph{not} necessarily turn cohomological Mackey 2-functors into cohomological Mackey 1-functors (see Remarks~\ref{Rem:decats} and~\ref{Rem:comack-gen-formula}).
To restore such a connection we propose in \Cref{sec:Hom-decat} a different type of decategorification, that works as follows:
\begin{Thm}[Hom-Decategorification]
\label{Thm:Hom-decat-intro}%
Let $\MM\colon \gpd^\op\to \ADD$ be a Mackey 2-functor. Let $G$ be a finite group and let $X,Y\in \MM(G)$ be two objects. Then there is an ordinary Mackey functor $M=M_{\MM,G,X,Y}$ for~$G$ whose value on every subgroup~$H\le G$ is given by the abelian group
\[
M(H)=\Homcat{\MM(H)}(\Res^{G}_H X\,,\,\Res^{G}_H Y).
\]
\end{Thm}

This is \Cref{Thm:Hom-decat-general} specialized to \Cref{Exa:local-general}.
We emphasize that this result works without the label `cohomological' and is thus of interest in the generality of~\cite{BalmerDellAmbrogio20}.
When $\MM$ is moreover cohomological then all its Hom-decategorifications are cohomological in the classical sense (\Cref{Thm:coh-decats}).

Several examples of Hom-decategorification are discussed in \Cref{sec:cohom-2-Mackey}.
For instance, group cohomology (the ur-example of a cohomological Mackey functor) is a Hom-decategorification of the derived category Mackey 2-functor $G\mapsto \Der(\kk G)$, itself a prime example of a cohomological Mackey 2-functor.

We also prove in \Cref{sec:cohom-2-Mackey} a descent result for cohomological Mackey 2-functors; see \Cref{Thm:p-local-rec}.

\tristars

Let us now turn our attention to Mackey motives.
Fix a commutative ring $\kk$ of coefficients, for instance~$\kk=\bbZ$. Suppose from now on that all our Mackey 2-functors take values in idempotent-complete $\kk$-linear additive categories $\MM(G)$ and $\kk$-linear functors~$u^*$. In~\cite{BalmerDellAmbrogio20}, we constructed a $\kk$-linear bicategory $\Motk$ (there denoted `$\kk \Spanhat$') of \emph{Mackey 2-motives}, together with a canonical contravariant embedding of the 2-category of finite groupoids
\begin{equation}
\label{eq:mot-intro}%
\mot\colon\gpd^\op \hooklongrightarrow \Motk.
\end{equation}
It enjoys the following universal property: Every Mackey 2-functor $\MM$ factors as $\MM\cong\widehat{\MM}\circ\mot$ for a unique $\kk$-linear pseudo-functor~$\widehat{\MM}$ on~$\Motk$.
The original construction of the bicategory $\Motk$ is pretty involved, with spans of 1-cells \emph{and} spans of 2-cells. (See \Cref{Rec:Mot}.) We prove in \Cref{sec:2-Mack-bisets} that there is a simpler description of~$\Motk$. Again, this holds beyond the cohomological world.
\begin{Thm}[Mackey 2-motives via bisets]
\label{Thm:new-Mackey-2-motives-intro}%
General $\kk$-linear Mackey 2-motives are modeled by the block-completion $(\kk \bisethatrf)^\flat$ of the bicategory $\kk \bisethatrf$ whose objects are finite groupoids, 1-morphisms are right-free finite bisets between them, and 2-morph\-isms are $\kk$-linear combinations of spans of equivariant maps.
\end{Thm}

Block-completion~$(-)^\flat$ is the 2-categorical analogue of idempotent-completion, a standard feature of motivic constructions. It simply adds formal summands, both at the 0- and 1-level, in order to split idempotent 2-cells; see \Cref{Rec:b}.

\Cref{Thm:new-Mackey-2-motives-intro} categorifies the usual equivalence between Webb's \emph{inflation functors} \cite{Webb93} and Bouc's \emph{right-free biset functors}~\cite{Bouc10}. See \Cref{Rem:Webbouc}.

\bigbreak

It is natural to look for an analogous motivic construction with \emph{cohomological} Mackey 2-functors.
Since these are just Mackey 2-functors satisfying some additional relations at the level of 2-cells, one can obtain the corresponding bicategory of \emph{cohomological Mackey 2-motives} $\CohMotk$ by formally modding out in~$\Motk$ the relevant 2-cells.
Thus formulated, $\CohMotk$ remains rather mysterious and one of our main goals is to give a simple computable description.
It comes as a categorification and generalization of Yoshida's Theorem \cite{Yoshida83b}; see \Cref{Rem:cf-Yoshida}.

\begin{Thm}[Cohomological Mackey 2-motives]
\label{Thm:univ-cohom-2-Mackey-intro}%
Cohomological $\kk$-linear Mackey 2-motives are modeled by the block-completion $(\bipermk)^\flat$ of the bicategory $\bipermk$ whose objects are finite groupoids, 1-morphisms are right-free permutation bimodules, and 2-morphisms are equivariant $\kk$-linear maps (see \Cref{Def:biperm-rfree}). In other words, there is a canonical pseudo-functor $\cohmot\colon\gpd^\op\to (\bipermk)^\flat=:\CohMotk$ sending a groupoid to itself and a functor $u\colon H\to G$ to the $G,H$-bimodule $\kk[G(u-,-)]\colon$ $H^\op\times G\to \Modcat(\kk)$, such that every $\kk$-linear cohomological Mackey 2-functor $\MM$ factors uniquely through this pseudo-functor up to isomorphism.
\[
\xymatrix@R=14pt{
\gpd^\op \ar[dr]_(.4){\cohmot} \ar[rr]^-{\MM} && {\ADDick}\\
 & \CohMotk\ar@{-->}[ur]_(.6){\exists !\,\widehat{\MM}}  &
 }
\]
\end{Thm}

A proof and more details on these constructions can be found in \Cref{sec:cohom-2-motives}. See in particular the 2-categorical universal property of~$\CohMotk$ in \Cref{Thm:UP-coh}.

\tristars

The organization of the paper is as follows. In the first part, we discuss general Mackey 2-functors, recalling the definitions in \Cref{sec:reminder}. We introduce Hom-decategorification in \Cref{sec:Hom-decat} and prove \Cref{Thm:Hom-decat-intro}. We revisit Mackey 2-motives in \Cref{sec:2-Mack-bisets} and prove \Cref{Thm:new-Mackey-2-motives-intro}.

In the second part of the paper, from \Cref{sec:cohom-2-Mackey} onwards, we focus on \emph{cohomological} Mackey 2-functors, starting with easy properties and the first examples and applications. Then, \Cref{sec:cohom-2-motives} contains the construction of \emph{cohomological Mackey 2-motives} and the proof of \Cref{Thm:univ-cohom-2-Mackey-intro}. In the short final \Cref{sec:motivic-decompositions}, we discuss motivic decompositions in $\Motk$ and $\CohMotk$ and we compare them in terms of certain explicit ring maps (\Cref{Thm:motivic-comparison}). As another application, we establish that each value category $\MM(G)$ of a cohomological Mackey 2-functor $\MM$ admits a canonical decomposition in terms of the blocks of the group algebra (\Cref{Thm:universal-block-factorization}).

\begin{Ack}
We are grateful to Baptiste Rognerud for several useful discussions. We thank the referee for suggesting how to streamline the paper.
\end{Ack}

%------------------------------------------------------------------------------
%------------------------------------------------------------------------------
%
\section{Recollections}
\label{sec:reminder}%
%\bigbreak
%------------------------------------------------------------------------------

We fix a commutative ring~$\kk$ throughout the article.

\begin{Ter}
We use the language of bicategories, 2-categories (\ie strict bicategories), pseudo-functors etc., in a standard way as recalled in \cite[App.\,A]{BalmerDellAmbrogio20}.
For simplicity, a \emph{$\kk$-linear} category means an \emph{additive} category enriched over $\kk$-modules.
Similarly, a \emph{$\kk$-linear} bicategory~$\cat{B}$ means one in which all Hom categories $\cat B(X,Y)$ and all horizontal composition functors are $\kk$-linear, and which admits finite direct sums of objects.
A pseudo-functor $\cat F$ is \emph{$\kk$-linear} if each functor $\cat F_{X,Y}$ is $\kk$-linear -- thus automatically preserves directs sums.
We call a contravariant pseudo-functor \emph{additive} if it turns coproducts of objects into products, as in \Mack{1} below.
\end{Ter}

\begin{Rec} \label{Rec:b}
An additive category is \emph{idempotent-complete} if every idempotent endomorphism $e= e^2$ on an object $X$ has an image, so that $X\cong \mathrm{Im}(e)\oplus \mathrm{Im}(1-e)$ identifying $e$ with the matrix $\mathrm{diag}(1,0)$.
Recall from \cite[\S\,A.7]{BalmerDellAmbrogio20} that a $\kk$-linear bicategory $\cat B$ is \emph{block-complete} if its Hom categories are idempotent-complete and if every decomposition of an identity 2-cell $\id_{\Id_X}$ in orthogonal idempotent 2-cells induces a direct sum decomposition of the object~$X$; so idempotent 2-cells split 1-cells and objects in directs sums.
Every $\kk$-linear bicategory~$\cat B$ admits a \emph{block-completion} \mbox{$\cat B\hookrightarrow \cat B^\flat$}, the universal $\kk$-linear pseudo-functor into a block-complete bicategory $\cat B^\flat$.
The precise construction of $\cat{B}^\flat$ is not difficult but requires some work. It can be found in~\cite[A.7.22]{BalmerDellAmbrogio20}. Suffice it to say that it is a categorification of Karoubi's classical idea to replace objects~$X$ by pairs $(X,e)$ where~$e=e^2\colon X\to X$ is an idempotent on~$X$, suitably adapted to the 2-categorical setting, \ie applied to both 0-cells (objects) and to 1-cells (as objects of the $\Hom$-categories).
\end{Rec}

\begin{Def} \label{Def:2Mack}
We recall that a \emph{\textup($\kk$-linear\textup) Mackey 2-functor} is the data of a 2-functor $\MM\colon \gpd^\op\to \ADD_\kk$ from the 2-category  of finite groupoids, functors and natural transformations to the 2-category of (possibly large) $\kk$-linear additive categories, additive functors and natural transformations.
It inverts the direction of 1-cells, so that we have a \emph{restriction} functor $u^*=\MM(u)\colon \MM(G)\to \MM(H)$ for every morphism (functor) of groupoids $u\colon H\to G$, and we have a natural isomorphism $\alpha^*\colon u^*\Rightarrow v^*$ for every natural isomorphism $u\Rightarrow v$. This is subject to four axioms:
\begin{enumerate}[\rm({Mack}\,1)]
\item
\label{Mack-1}%
\emph{Additivity:} $\MM(G_1\sqcup G_2) \overset{\sim}{\to} \MM(G_1)\times \MM(G_2)$ for all $G_1,G_2 \in \gpd$.
\item
\label{Mack-2}%
\emph{Adjoints:} For every faithful morphism $i\colon H\into G$, the restriction functor $i^*\colon \MM(G)\to \MM(H)$ admits a left adjoint $i_!$ and a right adjoint $i_*$.
\item
\label{Mack-3}%
\emph{Mackey formulas:} For every \emph{Mackey square} (a.k.a.\ pseudo-pullback, homotopy pullback; see \cite[Ch.\,2.1-2]{BalmerDellAmbrogio20}) as on the left-hand side below
\begin{equation}
\label{eq:Mackey-square}%
\vcenter{
\xymatrix@C=10pt@R=10pt{
& P \ar@{ >->}[rd]^j \ar[ld]_v \ar@{}[dd]|{\isocell{\gamma}} & \\
H \ar@{ >->}[rd]_i &  & K \ar[ld]^u \\
& G &
}}
\quad\quad
\rightsquigarrow
\quad\quad
\vcenter{
\xymatrix@L=5pt@R=12pt{
j_! \circ v^* \ar@{=>}[r]_-{\gamma_!}^-{\simeq} & u^* \circ i_! \\
u^* \circ i_* \ar@{=>}[r]_-{(\gamma^{-1})_*}^-{\simeq} & j_* \circ p^*
}}
\end{equation}
where $i$ and (thus)~$j$ are faithful, the two mates $\gamma_!$ and $(\gamma^{-1})_*$ for the adjunctions of \Mack{2} are both isomorphisms as displayed above.
\item
\label{Mack-4}%
\emph{Ambidexterity:} There exists a natural isomorphism $i_!\cong i_*$ between the left and right adjoints of every faithful~$i$.
\end{enumerate}
We may occasionally want to replace the source $\gpd$ of a Mackey 2-functor with a more general 2-category `of groupoids', \cf \Cref{Rem:more-general} or \cite[Hyp.\,5.1.1]{BalmerDellAmbrogio20}.
\end{Def}

\begin{Conv}
Unless otherwise stated, our Mackey 2-functors take values in the sub-2-category $\ADDick\subset \ADD_\kk$ of idempotent-complete additive categories. (In any case, every Mackey 2-functor can always be idempotent-completed termwise.)
\end{Conv}

\begin{Conv} \label{Rem:rectified}
By the Rectification Theorem \cite[Ch.\,3]{BalmerDellAmbrogio20}, the two adjunctions $i_! \dashv i^* \dashv i_*$ can be chosen to satisfy some extra properties.
For instance, we may take $i_!=i_*$ as functors (obvious from \Mack{4}).
Moreover, we may choose units and counits so that the base change isomorphisms in \Mack{3} are mutual inverses
\begin{equation}
\label{eq:2-Mack}%
(\gamma_!)\inv=(\gamma\inv)_*
\end{equation}
and so that the pseudo-functors $(-)_!$ and $(-)_*$ on $\JJ^\co$ induced by the left, resp.\ the right, adjunctions (\cite[A.2.10]{BalmerDellAmbrogio20}) are the same.
We will assume throughout all Mackey 2-functors to be \emph{rectified}, \ie to come with such (uniquely determined) superior choice of left and right adjunctions.
(We will assume this even in the more general setting of \Cref{sec:Hom-decat}, where the Rectification Theorem may not apply.)
Their units and counits will be denoted $\leta\colon \Id\Rightarrow i^*i_!$ and $\leps\colon i_!i^*\Rightarrow \Id$ (for the left adjunction $i_!\dashv i^*$) and $\reta\colon \Id \Rightarrow i_*i^*$ and $\reps\colon i^*i_*\Rightarrow\Id$ (for the right one $i^*\dashv i_*$).
\end{Conv}

\tristars

We make use of two closely related bicategories: that of \emph{bisets} and that of \emph{bimodules}. We briefly recall these well-known notions in order to establish notation.

\begin{Rec}[Bisets] \label{Rec:biset}
Let $G$ and $H$ be finite groupoids.
By a \emph{\textup(finite\textup) $G,H$-biset} $S= {}_GS_H$ we mean a functor $S\colon H^\op\times G\to \set$ to the category of finite sets.
We will often write
\[
g\cdot s = S(\id, g)(s)
\quad \textrm{ and } \quad
s \cdot h = S(h,\id) (s)
\]
to denote the `left action' of a morphism $g\in G(x,x')$ and the `right action' of a morphism $h \in H(y',y)$ on an element $s\in S(y,x)$.
We denote by $\biset$ the bicategory with finite groupoids as objects, all $G,H$-bisets ${}_G S_H$ as 1-cells $H\to G$, and all equivariant maps (\ie natural transformations) $\alpha\colon S\Rightarrow T$ as 2-cells.

The horizontal composition of bisets is provided by the tensor product of functors  (a.k.a.\ set-theoretic coends). Concretely, the value at $(z,x)\in K^\op\times G$ of a composite biset $({}_GT_{H})\circ ({}_HS_K) = T \times_H S$ is the following coequalizer of sets:
\begin{equation}
\label{eq:coend}%
\big(T\times_H S\big)(z,x)
 = \coeq \left (\coprod_{y'\to y\atop \in \,\Mor H} T(y,x) \times S(z,y') \rightrightarrows \!\coprod_{y\atop \in \,\Obj H} T(y,x) \times S(z,y) \right)\kern-.7em
\end{equation}
Even more concretely, an element of $(T\times_H S)(z,x)$ is the equivalence class of a pair $(t,s) \in T(y,x)\times S(z,y)$ for some $y\in \Obj H$, and two pairs $(t,s)$ and $(t',s')$ are equivalent if and only if there exists a morphism $h \in H(y,y')$ such that $(t,h\cdot s)= (t'\cdot h , s')$.
We will write $[t,s]$ for such a class, or sometimes $[t,s]_y$ or ``$[t,s]$~at~$y$'' if we need to keep track of the object $y\in \Obj H$.
The actions of $G$ and $K$ on $T\times_H S$ are the evident $g \cdot [t,s]=[g\cdot t,s]$ and $[t,s]\cdot k=[t,s\cdot k]$.

The identity biset of $G$ is the Hom-functor $\Id_G = G(-,-) \colon G^{\op}\times G\to \set$. (By defining bisets and bimodules (below) as functors on $H^\op\times G$, rather than the perhaps more common $G\times H^\op$, we think of them as `generalized Hom-functors'.)

A $G,H$-biset ${}_GS_{H}$ is \emph{right-free} if the right action of $H$ is free in the usual sense: $s\cdot h = s \Rightarrow h=\id$, for any $(y,x)\in H^\op\times G$, $s\in S(y,x)$ and $h\in H(y',y)$.
It is a straightforward exercise to see that the tensor product $T\times_H S$ of two right-free bisets is again right-free.
Thus right-free bisets form a 2-full sub-bicategory of $\biset$ which we denote $\biset^\rfree$.
\end{Rec}

\begin{Rem}
\label{Rem:bisets-agree}
For finite groups, seen as one-object groupoids, the notions of bisets and their composition in \Cref{Rec:biset} agree with those of Bouc~\cite{Bouc10}. The full subcategory of groups in the 1-truncation $\pih (\biset)$ is the ordinary category of bisets of \emph{loc.\,cit.}, and \emph{ditto} for right-free bisets. (\Cf \cite{DellAmbrogio22a}.)
\end{Rem}

\begin{Rec}[Bimodules]
\label{Rec:bimodules}
Let $G$ and $H$ be finite groupoids.
A \emph{$G,H$-bimodule} is a functor $M\colon H^\op\times G\to \Modcat(\kk)$.
We denote by $\Bimod_\kk$ the bicategory with finite groupoids as objects, $G,H$-bimodules as 1-cells $H\to G$, and equivariant maps as 2-cells.
The horizontal composition of bimodules is given by the usual tensor product, \ie $\kk$-linearly enriched coends (as in~\eqref{eq:coend} but taking the coequalizer in~$\Modcat(\kk)$)
\[
({}_G M_H)\circ({}_H N_K)=M\otimes_{\kk H}N,
\]
that we simply denote $M\otimes_{H}N$.
A \emph{\textup(finite\textup) permutation $G,H$-bimodule} is a $G,H$-bimodule which admits a finite $G,H$-invariant basis: There exist finite sets $S(y,x)$ for all $(y,x)\in H^\op\times G$ which are collectively stable under the $G$- and $H$-actions and such that each $S(y,x)$ is a basis of the (free) $\kk$-module~$M(y,x)$.

Permutation bimodules are closed under tensor products, hence form a sub-bicategory $\biperm_\kk$ of~$\Bimod_\kk$. Of course ordinary permutation $\kk G$-modules are simply the essential image of $G$-sets inside $\kk G$-modules, under $\kk$-linearization. Extended to bicategories, this takes the following form:
\end{Rec}

\begin{Prop}[Linearization]
\label{Prop:k-linearization}%
There is a well-defined and canonical pseudo-functor $\kk[-]\colon \biset\to \biperm_\kk$ mapping a groupoid~$G$ to itself, a $G,H$-biset~$U$ to the $G,H$-bimodule~$\kk[U]$ defined by taking the free $\kk$-module termwise: $(\kk[U])(y,x)=\kk[U(y,x)]$, and extending equivariant maps $\kk$-linearly.
\end{Prop}
\begin{proof}
This is a well-known phenomenon with groups and it extends to finite groupoids without a wrinkle. For horizontal functoriality, we use the canonical isomorphism $\kk[U\times_H V]\isoto \kk[U]\otimes_{H}\kk[V]$ for every $G,H$-biset~$U$ and $H,K$-biset~$V$, given on basis elements by $[u,v]\mapsto u\otimes v$. Details are left to the reader.
\end{proof}

\begin{Def}
\label{Def:biperm-rfree}%
We denote by $\perm^\rfree_\kk$ the 2-full sub-bicategory of $\biperm_\kk$ with the same objects (finite groupoids) and whose 1-cells are \emph{right-free} permutation bimodules, that is, those which belong to the essential image of the above linearization $\kk[-]$ restricted to right-free bisets. In other words, linearization restricts to a canonical pseudo-functor $\kk[-]\colon \biset^\rfree\to \perm^\rfree_\kk$.
Note that, for $H$ a group, a 1-cell ${}_GM_H$ in $\perm_\kk$ is right-free in this sense iff it is free as a right $\kk H$-module.
\end{Def}

%------------------------------------------------------------------------------
\section{Hom-decategorification}
\label{sec:Hom-decat}%
%------------------------------------------------------------------------------

In this section we show that any Mackey 2-functor provided with a coherent choice of a pair of objects in each of its value categories gives rise to an ordinary Mackey functor of Hom-groups.
This is a sort of `decategorification' procedure, distinct from the more usual $K_0$-style decategorification.

For future reference, we work here under more general hypotheses:

\begin{Hyp} \label{Hyp:appendix}
In this section, $(\GG;\JJ)$ denotes a \emph{spannable pair} as in \cite[\S3]{DellAmbrogio22a}, \ie an essentially small extensive  (2,1)-category with sufficiently many Mackey squares and coproducts with respect to a 2-subcategory~$\JJ$ closed under them and containing all equivalences. The examples to keep in mind are $(\gpd; \gpdf)$, used in the other sections of this article, and $(\gpdG;\gpdG)$, used in this section; but $\GG$ does not necessarily consist of groupoids nor $\JJ$ of faithful 1-morphisms.
The definition of a Mackey 2-functor still makes immediate sense for a general spannable pair $(\GG;\JJ)$.
\end{Hyp}

\begin{Def} [Mackey 1-functors]
\label{Def:ordinary-MF}
The point of \Cref{Hyp:appendix} is that it allows us to define a \emph{span category} $\Spancat(\GG;\JJ):= \pih \Span( \GG; \JJ )$ which is semi-additive, \ie equipped with finite biproducts induced by the coproducts of objects in~$\GG$.
Then we can define an \emph{\textup(ordinary\textup)  Mackey \textup(1-\textup)functor for $(\GG;\JJ)$} to be an additive functor
\[
M\colon \Spancat(\GG;\JJ) \to \Ab
\]
to the category of abelian groups.
Concretely (\cf \cite[\S3]{DellAmbrogio22a}), a Mackey 1-functor $M$ for~$(\GG;\JJ)$ consists of an abelian group~$M(G)$ for every $G\in\Obj (\GG)$ together with a \emph{restriction} homomorphism $u^\sbull\colon M(G)\to M(H)$ for every $u\colon H\to G$ in~$\GG$ and a \emph{transfer} homomorphism $u_\sbull\colon M(H)\to M(G)$ when moreover $u\in \JJ$; this data must satisfy the following axioms: %
\begin{enumerate}[\rm(A)]
\item
\label{it:1-Mackey-fun}%
\emph{Functoriality:}
We have $\id^\sbull=\id_\sbull=\id$.
For all $K\xto{v} H\xto{u}G$, we have $(u\circ v)^\sbull=v^\sbull\circ u^\sbull$; and also $(u\circ v)_\sbull=u_\sbull\circ v_\sbull$ if they belong to~$\JJ$.
\item
\label{it:1-Mackey-iso}%
\emph{Isomorphism invariance:} For every 2-isomorphism $\alpha \colon u \overset{\sim}{\Rightarrow} v$ in~$\GG$, we have $u^\sbull = v^\sbull$; and also $u_\sbull = v_\sbull$ if they belong to~$\JJ$.
\item
\label{it:1-Mackey-add}%
\emph{Additivity:}
$M(\emptyset) \cong 0$ and every coproduct $G\overset{i}{\to} G\sqcup H \overset{\;j}{\gets} H$ in~$\GG$ yields an isomorphism
$(i^\sbull, j^\sbull)^t \colon M(G\sqcup H) \overset{\sim}{\to} M(G)\oplus M(H)$ with inverse $(i_\sbull , j_\sbull)$.
\item
\label{it:1-Mackey-formula}%
\emph{Mackey formula:}
For every Mackey square in~$\GG$ with $i$ and $q$ in~$\JJ$
\begin{equation}
\label{eq:Mackey-square-in-1-Mackey}%
\vcenter{
\xymatrix@C=10pt@R=10pt{
& P \ar[ld]_-{p} \ar[rd]^{q} \ar@{}[dd]|{\isocell{\gamma}} & \\
H \ar[rd]_-{i} &  & K \ar[ld]^-{u} \\
& G &
}}
\end{equation}
we have $u^\sbull\circ i_\sbull=q_\sbull\circ p^\sbull$.
\end{enumerate}
\end{Def}

\begin{Exas}
\label{Exas:ordinary-MF}
By specializing \Cref{Def:ordinary-MF} to various choices of $(\GG;\JJ)$ we obtain several classical variations on the notion of Mackey functor.
For instance, $(\gpdf; \gpdf)$ provides the so-called \emph{globally defined Mackey functors}.
The choice $(\gpd; \gpdf)$ results in the (global) \emph{inflation functors}.
Taking $\GG = \JJ =  \gpdG \cong G\sset$ to be groupoids faithfully embedded in a fixed group~$G$, we get the original `$G$-local' notion of \emph{Mackey functors for~$G$}.
See \cite{DellAmbrogio22a} for explanations and details.
\end{Exas}

\begin{Def}[Coherent family of pairs]
\label{Def:coherent-pair}
Let $\MM\colon \GG^\op \to \ADD$ be a Mackey 2-functor for $(\GG;\JJ)$. By a \emph{coherent family of pairs of objects in~$\MM$} we mean a 2-functor
$\cat{M}'' \colon \GG^\op\to \ADD''$
which lifts~$\MM$ along the forgetful 2-functor $\ADD''\to \ADD$, where we write $\ADD'':= (\mathbb Z\ffree \sqcup \mathbb Z\ffree)/\ADD$ for the (pseudo) slice 2-category of $\ADD$ under a coproduct of two copies of the free additive category on one object. Concretely, such a 2-functor $\MM''$ amounts to the following data:
\begin{enumerate}[\rm(1)]
\setcounter{enumi}{-1}
\item \label{it:family-0}
 two objects $X_G,Y_G\in \MM(G)$ for every object $G\in \GG$ and
\item \label{it:family-1}
two isomorphisms $\lambda_u\colon X_H \overset{\sim}{\to} u^*X_G$ and $\rho_u\colon u^*X_G \overset{\sim}{\to} Y_H$ in the category $\MM(H)$ for every morphism $u\colon H\to G$ in~$\GG$,
\end{enumerate}
satisfying the following conditions:
\begin{enumerate}[\rm(1)]
\setcounter{enumi}{1}
\item \label{it:family-2}
the triangles
\begin{equation*}
\vcenter{
\xymatrix@R=4pt{
& u^*X_G \ar[dd]^-{\alpha_{X_G}}
\\
X_H \ar[ru]^-{\lambda_u} \ar[rd]_-{\lambda_v} & \\
& v^*X_G
}}
\qquadtext{and}
\vcenter{
\xymatrix@R=4pt{
u^*Y_G \ar[rd]^-{\rho_u} \ar[dd]_-{\alpha_{Y_G}}
& \\
& Y_H \\
v^*Y_G \ar[ru]_-{\rho_v}
}}
\end{equation*}
commute for every 2-morphism $\alpha\colon u\Rightarrow v \colon H\to G$ of~$\GG$,
\item \label{it:family-id}
the equations $\lambda_{\Id_G} = \id_{X_G}$ and $\rho_{\Id_G} = \id_{Y_G}$ hold for every object $G\in \GG$,
\item \label{it:family-fun}
and finally, the triangles
\[
\vcenter{
\xymatrix@R=8pt{
X_K \ar@/_2ex/[drr]_-{\lambda_{uv}} \ar[r]^-{\lambda_v} & v^* X_H \ar[r]^-{v^*(\lambda_u)} & v^* u^* X_G \ar@{=}[d] \\
& & (uv)^* X_G
}}
\quad\textrm{and}\quad
\vcenter{
\xymatrix@R=8pt{
v^*u^* Y_G \ar@{=}[d] \ar[r]^-{v^*(\rho_u)} & v^* X_H \ar[r]^-{\rho_v} & X_K \\
(uv)^*X_G \ar@/_3pt/[urr]_-{\rho_{uv}} & &
}}
\]
commute for every composable pair of 1-morphisms $K \overset{v\,}{\to} H \overset{u\,}{\to} G$ of~$\GG$,
\end{enumerate}
\end{Def}

\begin{Rem}
\label{Rem:autom-add}
Any such lift $\MM''$ is automatically additive, since the forgetful 2-functor $\ADD''\to \ADD$ creates direct sums of objects in the evident way.
\end{Rem}

\begin{Thm} [Hom-decategorification]
\label{Thm:Hom-decat-general}
Let $\MM$ be a Mackey 2-functor for $(\GG;\JJ)$ (as in \Cref{Hyp:appendix}) and let $\MM''$ be a coherent family of pairs of objects in~$\MM$ as in \Cref{Def:coherent-pair}, given by $\{ X_G, Y_G, \lambda_u , \rho_u \}_{G,u}$.
Then there exists a Mackey 1-functor for $(\GG;\JJ)$  (\Cref{Def:ordinary-MF})
\[
M := M_{\MM''} = M_{\{ X_G, Y_G, \lambda_u , \rho_u \}} \colon \pih ( \Span (\GG; \JJ)) \longrightarrow \Ab
\]
whose values are given by the  Hom group at the chosen pair
\[
M(G) := \MM(G)(X_G, Y_G)
\]
for all objects $G\in \GG$, with restriction maps (obviously) defined by
\[
\xymatrix@R=2pt{
u^\sbull  \colon  M(G) = \MM(G)(X_G,Y_G) \ar[r] & \MM(H)(X_H, Y_H) = M(H) \\
 \quad \quad f \ar@{|->}[r] & \rho_u \circ u^*(f) \circ \lambda_u
}
\]
for all $u\colon H\to G$
and induction maps (less obviously) defined by
\[
\xymatrix@R=2pt{
i_\sbull \colon  M(H) = \MM(H)(X_H,Y_H) \ar[r] & \MM(G)(X_G, Y_G) = M(G) \\
\quad \quad g \ar@{|->}[r] & \leps \circ i_*(\rho_i^{-1} g \lambda_i^{-1} ) \circ \reta
}
\]
for all $i\colon H\to G$ in~$\JJ$ where $\reta$ and $\leps$ are the (co)units of~\Cref{Rem:rectified}.
\end{Thm}

\begin{Exa} [The $G$-local case]
\label{Exa:local-general}
Suppose $\MM$ is a $G$-local Mackey 2-functor, \ie a Mackey 2-functor for $\GG=\JJ = \gpdG \cong G\sset$ where $G$ is a fixed finite group. Then any pair of objects $X_0,Y_0\in \MM(G)$ defines a coherent choice as in the theorem, by setting $X_{(H,i_H)}:= i_H^* X_0$ and $Y_{(H,i_H)} := i_H^* Y_0$ for every object $(H, i_H \colon H\rightarrowtail G)$ of $\gpdG$, and $\lambda_{(u,\theta)} := (\theta^*)^{-1}\colon i_H^*X_0 \to u^*i_K^*X_0$ and $\rho_{(u,\theta)} := \theta^*\colon u^*i_K^* Y_0 \to i_H^*X$
for every morphism $(u\colon H\to K, \theta\colon i_K u \overset{\sim}{\Rightarrow} i_H)$.
We thus obtain an ordinary Mackey functor $M$ for~$G$ in the classical sense such that $M(H) = \MM(H)(\Res^G_H X_0, \Res^G_H Y_0)$ for all subgroups $H\leq G$.
\end{Exa}

\begin{Exa}
\label{Exa:general-monoidal}
Suppose that $\MM$ is a Mackey 2-functor for $(\GG;\JJ)$ taking values in monoidal categories $\MM(G)$ and strong monoidal functors~$u^*$ (for instance $\MM$ could be a \emph{Green 2-functor} in the sense of \cite{DellAmbrogio21pp}).
Then we may take $X_G = Y_G := \mathbf 1$ to be the tensor unit of $\MM(G)$ and $\lambda_u, \rho_u$ to be the coherent isomorphisms of~$u^*$.
This produces a Mackey functor $M$ for $(\GG;\JJ)$ with $M(G)= \End_{\MM(G)}(\mathbf 1)$.
In the presence of tensor-compatible gradings, \eg in the case of tensor triangulated categories, we also have a graded version $G\mapsto M(G)= \End_{\MM(G)}^*(\mathbf 1)$.
\end{Exa}

\begin{proof}[Proof of \Cref{Thm:Hom-decat-general}]
The restriction and transfer maps defined in the theorem are clearly additive, and we need to show that they satisfy the axioms \eqref{it:1-Mackey-fun}-\eqref{it:1-Mackey-formula} of \Cref{Def:ordinary-MF}.
The additivity axiom \eqref{it:1-Mackey-add} is immediate from \Cref{Rem:autom-add}.
Isomorphism invariance \eqref{it:1-Mackey-iso} is an easy consequence of \Cref{Def:coherent-pair}~\eqref{it:family-2}.

Let us check~\eqref{it:1-Mackey-fun}.
For $K\overset{j}{\to} H \overset{i}{\to} G$ in~$\JJ$ and $g\in M(K)=\MM(K)(X_K, Y_K)$, the following diagram (where $\eta = \reta$ and $\varepsilon = \leps$) commutes:
\[
\xymatrix@C=18pt@R=1.8em{
i_*X_H \ar[d]_{\simeq}^{i_*\lambda_i}
  \ar[r]^-{i_*\eta} &
i_*j_*j^* X_H
  \ar[d]|{i_*j_*j^*\lambda_i} &
i_*j_* X_K
  \ar[l]_-{i_*j_*\lambda_j}^-{\simeq}
  \ar[r]^-{i_*j_* g}
  \ar@{=}[d] &
i_*j_* Y_K
  \ar@{=}[d] &
i_*j_*Y_H
  \ar[l]_-{i_*j_*\rho_j}^-{\simeq}
  \ar[r]^-{i_*\varepsilon} &
i_*Y_H \\
i_*i^* X_G
  \ar[r]^-{i_* \eta i^*} &
 i_*j_*j^*i^* X_G
  \ar[d]_{\simeq} &
i_*j_* X_K
  \ar[l]^-{i_*j_* \lambda_{ij}}_-{\simeq}
  \ar@{}[ul]|{\textrm{\eqref{it:family-fun}}}
  \ar[d]_-{\simeq} &
i_*j_* Y_K
  \ar[d]^{\simeq} &
i_*j_*j^*i^* Y_G
  \ar[l]_-{\simeq}^-{i_*j_* \rho_{ij}}
  \ar[d]^{\simeq}
  \ar[u]|{i_*j_*j^*\rho_i}
  \ar@{}[ul]|{\textrm{\eqref{it:family-fun}}}
  \ar[r]_-{i_* \varepsilon i^*} &
i_*i^* Y_G
  \ar[u]^{i_* \rho_i}_{\simeq}
  \ar[d]^\varepsilon \\
X_G
  \ar[u]^{\eta}
  \ar[r]_-{\eta} &
(ij)_* (ij)^* X_G &
(ij)_* X_K
  \ar[l]^-{(ij)_* \lambda_{ij}}
  \ar[r]_-{(ij)_* g} &
(ij)_* Y_K
  \ar[r]^-{\simeq}_-{(ij)_* \rho_{ij}} &
(ij)_*(ij)^* Y_G
  \ar[r]_-{\varepsilon} &
Y_G
}
\]
The squares marked~\eqref{it:family-fun} commute by \Cref{Def:coherent-pair}.
The top-left and top-right ones commute by the naturality of $\reta$ and $\leps$. The isomorphism $i_*j_*\overset{\sim}{\to}(ij)_*$ is the pseudo-functoriality of $(-)_*=(-)_!$ (equivalently obtained from that of $(-)^*$ by taking either right or left mates---see \Cref{Rem:rectified}), and the five remaining squares commute by its basic properties; see~\cite[A.2.10]{BalmerDellAmbrogio20}.
The two paths connecting $X_G$ to $Y_G$ around the perimeter display the identity $(i_\sbull \circ j_\sbull) (g)= (ij)_\sbull (g)$. The remaining identities are easier and left to the reader.

It only remains to check the Mackey formula~\eqref{it:1-Mackey-formula}.
We want to show that
\[
u^\sbull i_\sbull=q_\sbull p^\sbull\colon \MM(H)(X_H,Y_H)\longrightarrow \MM(K)(X_K,Y_K).
\]
Let $f\colon X_H\to Y_H$ be a morphism in~$\MM(H)$. Its image under~$u^\sbull i_\sbull$ is the composite map $X_K\to Y_K$ over the top row of the following diagram, whereas its image under $q_\sbull p^\sbull$ is the composite over the bottom row:
\[
\xymatrix@C=9pt@R=1.5em{
u^*X_G
  \ar[rr]_-{u^* \reta} &&
u^*i_*i^* X_G  &&
u^*i_* X_H
  \ar[ll]_-{\simeq}^-{u^*i_* \lambda_i}
  \ar@<-2ex>[dd]_{(\gamma^{-1})_*}
  \ar@{}[dd]|{{\simeq}}
  \ar[rr]_-{u^*i_* (f)}  &&
u^*i_* Y_H  &&
u^*i_*i^* Y_H
  \ar[ll]_-{\simeq}^-{u^*i_* \rho_i}
  \ar[rr]_-{u^* \leps} &&
u^*Y_G
  \ar[d]_{\rho_u}^-{\simeq}
\\
X_K
  \ar[u]_{\lambda_u}^-{\simeq}
  \ar[d]^{\reta}  && && && && &&  Y_K \\
q_*q^* X_P  &&
q_*X_P
  \ar[ll]_-{q_*\lambda_q}^-{\simeq}
  \ar[rr]^-{q_* \lambda_p}_-{\simeq} &&
q_*p^* X_H
  \ar[rr]^-{q_*p^* (f)}
  \ar@<-2ex>[uu]_{\gamma_!} &&
q_*p^* Y_H
  \ar[rr]^-{q_* \rho_p}_-{\simeq}
  \ar[uu]_{\gamma_!} &&
q_*Y_P &&
q_*q^* Y_P
  \ar[ll]_-{q_*\rho_q}^-{\simeq}
  \ar[u]^{\leps}
}
\]
We insert in this diagram the 2-Mackey isomorphism $(\gamma\inv)_*=(\gamma_!)^{-1}\colon u^*i_* \isoTo q_*p^*$ of~\eqref{eq:2-Mack}. Since $\gamma_!$ is natural, the middle square above commutes.
It remains to show that the left and right heptagons commute.
The left heptagon is the perimeter of the following commutative diagram (where $\eta = \reta$ and $\varepsilon = \reps$):
\[
\xymatrix@C=36pt{
X_K
  \ar[d]_\eta
  \ar[r]^-{\lambda_u} &
u^* X_G
  \ar[d]_\eta
  \ar[r]^-{u^* \eta} &
u^*i_*i^* X_G
  \ar[d]_\eta &
u^*i_* X_H
  \ar[d]_\eta
  \ar[l]^-{\simeq}_-{u^*i_* \lambda_i}
  \ar@/^8ex/[ddd]^{(\gamma^{-1})_*} & \\
q_*q^* X_K
  \ar[r]^-{q_*q^* \lambda_u}
  \ar@{}[dr]|(.3){\textrm{\eqref{it:family-fun}}}
  \ar@{}[dr]|(.7){\textrm{\eqref{it:family-2}}} &
q_*q^*u^* X_G
  \ar[r]^-{q_*q^*u^* \eta}
  \ar[d]^{q_* (\gamma^{-1})^* } &
q_*q^*u^*i_*i^* X_G
  \ar[d]^{q_* (\gamma^{-1})^* i_*i^*} &
q_*q^*u^*i^* X_H
  \ar[l]^-{\simeq}_-{q_*q^*u^*i_* \lambda_i}
  \ar[d]|{q_*(\gamma^{-1})^* i_*} & \\
q_*X_P
  \ar[u]^{q_*\lambda_q}_{\simeq}
  \ar[ur]|{q_* \lambda_{uq}}
  \ar[r]_-{q_*\lambda_{ip}}
  \ar@/_3ex/[dr]_-{q_* \lambda_p}
  \ar@{}[dr]|(.6){\textrm{\eqref{it:family-fun}}} &
q_*p^* i^* X_G
   \ar[r]^-{q_*p^*i^* \eta}
   \ar@{=}[dr]   &
q_*p^*i^*i_*i^* X_G
  \ar[d]^{q_*p^* \varepsilon  i^*}
  \ar@{}[dl]^(.25){(i^*\dashv i_*)}  &
q_*p^* i^* i_* X_H
  \ar[l]^-{\simeq}_-{q_*p^*i^*i^* \lambda_i }
  \ar[d]_{q_*p^* \varepsilon} & \\
&
q_*p^* X_H
  \ar[u]_{q_*p^*\lambda_i}
  \ar@/_4ex/@{=}[rr] &
q_*p^* i^* X_G &
q_*p^* X_H
  \ar[l]^-{\simeq}_-{q_*p^* \lambda_i} & \\
  & & & &
}
\]
The latter commutes by hypotheses \eqref{it:family-2} and~\eqref{it:family-fun} in \Cref{Def:coherent-pair}, a zig-zag equation for the adjunction $i^*\dashv i_*$, the definition of the mate~$(\gamma^{-1})_*$, and the naturality of various maps.
The right heptagon is analogous and is left to the reader.
\end{proof}

%------------------------------------------------------------------------------
%------------------------------------------------------------------------------
%
\section{Mackey 2-motives via bisets}
\label{sec:2-Mack-bisets}%
%\bigbreak
%------------------------------------------------------------------------------

In this section we prove \Cref{Thm:new-Mackey-2-motives-intro}. Let us begin by recalling the original construction of Mackey 2-motives in~\cite{BalmerDellAmbrogio20}.

\begin{Rec} \label{Rec:Mot}
Mackey 2-motives can be constructed in four steps:
\begin{equation}
\label{eq:mot}%
\mot\colon \gpd^\op
\longrightarrow
\Span^\rfree
\longrightarrow
\Spanhatrf
\longrightarrow
\kk \Spanhatrf
\longrightarrow
(\kk \Spanhatrf)^\flat=:\Motk
\end{equation}
One begins by building the bicategory $\Span^\rfree$, where objects are finite groupoids, 1-cells $H\to G$ are \emph{right-faithful spans} in~$\gpd$
\[
H \overset{u}{\gets} P \overset{i}{\rightarrowtail} G
\]
consisting of a functor $u\colon P\to H$ and a faithful functor $i\colon P\into G$ with common source.
For 2-cells we use isomorphism classes (in the standard sense) of triples
\begin{equation}
\label{eq:2-mor-Span}%
[p,\alpha,\beta]=\quad
\vcenter{\xymatrix@C=5em@R=.5em{
& P \ar[dl]_u \ar[dr]^i \ar[dd]^-{p} & \\
G \ar@{}[r]|(.6){\alpha \SEcell} & & H \ar@{}[l]|(.6){\NEcell \beta}\\
& Q \ar[ul]^v \ar[ru]_j &
}}
\; .
\end{equation}
Horizontal composition is computed by forming Mackey squares (see \cite[Ch.\,5]{BalmerDellAmbrogio20}).
In the next step of~\eqref{eq:mot}, we enlarge $\Span^\rfree$ to a bicategory $\Spanhatrf$ by also allowing the formation of spans vertically, \ie spans of 2-cells of $\Span^\rfree$ (see \cite[Ch.\,6]{BalmerDellAmbrogio20}).
The bicategory $\Spanhatrf$ is locally \emph{semi}-additive, \ie its Hom categories admit finite biproducts and thus are canonically enriched in abelian monoids.
For the next step, we group-complete all Hom monoids of 2-cells and tensor them with~$\kk$ to obtain a $\kk$-linear bicategory $\kk \Spanhatrf$.
Finally, we define $\Motk$ to be  the block-completion $(-)^\flat$ of $\kk \Spanhatrf$ (see \Cref{Rec:b}).

At each step, we have an evident canonical pseudo-functor as pictured in~\eqref{eq:mot} above, starting with the contravariant embedding $(-)^*\colon\gpd^\op\to \Span^\rfree$ sending a functor $u\colon H\to G$ to the span $G \overset{u}{\gets} H \xto{\Id} H$ and a natural isomorphism $\alpha \colon u\Rightarrow v$ to the morphism of spans represented by the triple~$[\Id_H,\alpha,\id_{\Id_H}]$.
\end{Rec}

\begin{Warning}
Our present notations differ slightly from \cite{BalmerDellAmbrogio20}. There $\Span^\rfree$ was denoted $\Span(\gpd;\gpdf)$ or simply $\Span$, and similarly for $\Spanhatrf$.
The symbol $\kk \Spanhat$ was previously used to directly denote $\Motk$, including block-completion.
\end{Warning}

\begin{Rem} \label{Rem:more-general}
\Cref{Def:2Mack} and \Cref{Rec:Mot} work for more general `(2,1)-categories of groupoids'~$\GG$ and more general classes of faithful 1-morphisms~$\JJ$, leading to variants~$\Motk(\GG;\JJ)$ of the motivic bicategory. Everything in this article generalizes too but this will be ignored for simplicity (see Examples~\ref{Exa:linear} and~\ref{Exa:geom}).
\end{Rem}

\begin{Not} \label{Not:bimacks}
As in \cite[\S\,6.3]{BalmerDellAmbrogio20}, idempotent-complete $\kk$-linear Mackey 2-functors, together with `induction preserving' morphisms and modifications, form a 2-category here denoted by $\bicMackk$.
It is contained in a 2-category $\biMackk$ of all, non-necessarily idempotent-complete, $\kk$-linear Mackey 2-functors, which is itself contained in the 2-category
$\PsFun_\amalg (\gpd^\op, \ADD_\kk)$ of all additive (\ie coproduct-preserving) pseudo-functors, pseudo-natural transformations and modifications.
\end{Not}

Here is the universal property of Mackey 2-motives $\Motk$:

\begin{Thm}[{\cite[\S\S\,5.3 and 6.3]{BalmerDellAmbrogio20}}]
\label{Thm:UP-general}
The canonical pseudo-functors of \Cref{Rec:Mot} induce by precomposition biequivalences of 2-categories
\[
\PsFun_\kk (\kk\Spanhatrf, \ADD_\kk) \overset{\sim}{\to} \biMackk
\quad \textrm{and} \quad
\PsFun_\kk (\Motk, \ADDick) \overset{\sim}{\to} \bicMackk
\]
where $\PsFun_\kk$ denotes the 2-category of $\kk$-linear (hence additive) pseudo-functors, pseudo-natural transformations and modifications.
\end{Thm}

\tristars

\begin{Rem}
\label{Rem:clash}%
With these reminders behind us, \Cref{Thm:new-Mackey-2-motives-intro} tells us that in the construction of Mackey 2-motives we may replace right-faithful spans of functors with right-free bisets (see \Cref{Rec:biset} for the latter).
A span $H \gets P \rightarrowtail G$, from $H$ to $G$, is \emph{right}-faithful when $P\to G$ is faithful.
However, a $G,H$-biset ${}_GS_H$, still from $H$ to~$G$, is \emph{right}-free when the $H$-action is free, not the $G$-action.
So the meticulous reader might be puzzled that we use the same decoration `$\rfree$' both for ``right-faithful'' and ``right-free'' in apparently unrelated cases.
In fact, the following key result shows that these two notions match beautifully.
\end{Rem}

\begin{Thm}
\label{Thm:Span-vs-biset}
There exists a canonical biequivalence of bicategories
\[
\xymatrix@1{
{\Span^\rfree} \ar@{}[rr]|-{\sim} \ar@<.8ex>[rr]^-{\cat R} && {\biset^\rfree} \ar@<.8ex>[ll]^-{\int}
}
\]
given by the \emph{realization bifunctor}~$\cat R$ of spans (see \Cref{Rec:realization}) and the \emph{Grothendieck construction} $\int$ on bisets (see \Cref{Rec:Groth-constr}). On objects, \ie finite groupoids, both pseudo-functors are just the identity.
\end{Thm}

\begin{Rem}
Results closely related to the above one have long been known among some category-theorists (see \eg \cite{Benabou00}) and topologists (see \eg \cite{Miller17}).
\end{Rem}

\begin{proof}[Proof of \Cref{Thm:new-Mackey-2-motives-intro}]
The biequivalence between $\Motk = (\kk \Spanhatrf)^\flat$ and $(\kk \bisethatrf)^\flat$ is easily obtained from that of \Cref{Thm:Span-vs-biset} by changing both sides as follows:
\begin{enumerate}[\rm(1)]
\item take ordinary categories of spans $\cat C\mapsto \widehat{\cat C}$ of all Hom categories (\cite[A.4]{BalmerDellAmbrogio20});
\item group-complete every Hom abelian monoid of 2-cells;
\item extend scalars from $\bbZ$ to~$\kk$;
\item and finally, take block-completions~$(-)^\flat$.
\end{enumerate}
Each operation is sufficiently bifunctorial to preserve biequivalences.
\end{proof}

In order to prove \Cref{Thm:Span-vs-biset}, we first detail the constructions of $\cat R$ and~$\int$.

\begin{Rec}[The realization pseudo-functor]
\label{Rec:realization}
Let us first consider the bicategory $\Span:=\Span(\gpd)$ of \emph{all}, not necessarily right-faithful, spans between finite groupoids, as well as the bicategory $\biset$ of \emph{all}, not necessarily right-free, bisets (\Cref{Rec:biset}). By \cite[\S4.2]{Huglo19pp} or~\cite{DellAmbrogioHuglo21}, there is a pseudo-functor
\begin{equation} \label{eq:full-realization}
\cat R \colon \Span \too \biset
\end{equation}
which sends a finite groupoid $G$ to itself, a span of functors $i_!u^*:=( H \xleftarrow{u} P \xto{i} G)$ to the composite biset
\[
\cat R(i_!u^*)
:= \underbrace{G(i-,-)}_{{\cat R_!(i)}} \times_P \underbrace{H(-,u-)}_{{\cat R^*(u)}} \colon H^\op\times G \too \set
\]
and morphisms of spans~\eqref{eq:2-mor-Span} to the naturally induced morphism of bisets.
Note that $\cat R$ is obtained by the universal property of $\Span$ (see \cite[\S\,5.2]{BalmerDellAmbrogio20} or \cite[Thm.\,5.4]{DellAmbrogioHuglo21}) by `gluing' the two more evident pseudo-functors
\[ \cat R^*\colon \gpd^\op\too \biset \quad \textrm{ and } \quad \cat R_!\colon \gpd^\co \too \biset \]
which map a functor $v\colon P\to Q$ to the biset $\cat R^*(v)=Q(-,v-)\colon Q^\op\times P\to \set$, respectively
to the biset $\cat R_!(v)= Q(v-,-) \colon P^\op \times Q\to \set$.
This gluing is possible because in $\gpd$ there are (well-behaved) adjunctions for every $v\colon P\to Q$
\begin{equation}
\label{eq:realization-adj}%
\vcenter{
\xymatrix@R=2em{
P
\ar@/_2ex/[d]_{\cat R_!(v)\,=\,Q(v-,-)}
\ar@/^2ex/@{<-}[d]^{Q(-,v-)\,=\,\cat R^*(v)}
 \ar@{}[d]|{\dashv} \\
Q
}}
\end{equation}
with unit $\eta\colon \Id_P \Rightarrow \cat R^*(v) \circ \cat R_! (v) $ and counit $\varepsilon\colon \cat R_!(v)\circ \cat R^*(v)\Rightarrow \Id_Q$  given by
\begin{align*}
\eta_{x,x'} \colon & \;\; P(x,x') \too Q(z,v(x'))\times_{z\in Q} Q(v(x),z)\,, &&  p \mapsto  [\id_{v(x')}, v(p)] \\
\varepsilon_{y,y'}  \colon & \;\; Q(v(z),y')\times_{z\in P} Q(y,v(z)) \too Q(y,y') \,, &&  [ q_1, q_2 ]  \mapsto q_1 q_2
\end{align*}
for all $x,x' \in \Obj(P)$ and $y,y' \in \Obj(Q)$.

Note that the \emph{data} of the pseudo-functor $\cat R$ is entirely determined by the data of the above pseudo-functors $\cat R_!$ and $\cat R^*$ and the adjunctions $(\cat R_!(v), \cat R^*(v), \eta, \varepsilon)$.
\end{Rec}

\begin{Rem}
The realization pseudo-functor \eqref{eq:full-realization} is not a biequivalence, as can already be see at the level of truncated 1-categories.
Indeed, the resulting functor $\pih (\cat R)\colon \pih(\Span)\to \pih(\biset)$ is full but not faithful, and its kernel admits a nice description due to Ganter and Nakaoka (see \cite[\S6]{DellAmbrogio22a}).
In order to get a biequivalence, we must restrict both its domain and codomain.
\end{Rem}

\begin{Lem} \label{Lem:rf}
Let $i_!u^*=(H \xleftarrow{u} P \xrightarrow{i} G)$ be a span of finite groupoids.
If the functor $i$ is faithful, then the biset $\cat R(i_!u^*)$ is right-free, \ie $H$ acts freely on it.
\end{Lem}

\begin{proof}
By definition, the biset $\cat R(i_!u^*)$ is right-free if and only if for every objects $(y,x)\in \Obj(H^\op\times G)$ and $y'\in \Obj(H)$, every element $[g, h]\in \cat R(i_!u^*)(y,x)$ and every morphism $t\in H(y',y)$, we have that $[g,h]\cdot t=[g, h]$ implies $t=\id_y$ (note that we must already have $y'=y$ for the first equation to make sense).
Here $g \in G(iz,x)$ and $h \in H(y,uz)$ for some $z\in \Obj(P)$ and $[g, h]\cdot t= [g , ht]$ by definition.
Thus the equation $[g,h]\cdot t=[g, h]$ means that there exists some map $p\in P(z,z)$ such that $g \circ i(p) = g$ and $u(p)\circ h = h\circ t$.
As $G$ is a groupoid, the first equation entails $i(p)=g^{-1}g = \id_{i(z)}$.
If $i$ is faithful, the latter implies that $p=\id_y$ and thus, by the second equation $u(p) h = h t$ we get $t = h^{-1}h = \id_y$ as wished.
\end{proof}

In the other direction, we use the following construction:

\begin{Rec}[Grothendieck construction]
\label{Rec:Groth-constr}
Fix two groupoids $H$ and $G$. For any $G,H$-biset $S\in \biset (H,G)$, we can define a groupoid denoted
\[
\textstyle\int_G^H S
\quad \textrm{ or simply } \quad
\textstyle\int S
\]
whose objects are triples $(y,x,s)$ with $y\in \Obj H^\op$, $x\in \Obj G$ and $s\in S(y,x)$.
A morphism $(y,x,s)\to (y',x',s')$ in $\int S$ is a pair of morphisms $(h,g)$ with $h\in H(y,y')$ and $g\in G(x,x')$, such that $g \cdot s = s' \cdot h $ holds.
This comes equipped with obvious projection functors $\pr_H\colon \int S\to H $ and $ \pr_G \colon \int S\to G$, sending $(y,x,s)$ and $(h,g)$ to $y$ and~$h$, respectively to $x$ and~$g$.
In other words, we obtain a span from $H$ to~$G$
\[
\xymatrix@R=5pt{
&{\int S} \ar[dl]_{\pr_H} \ar[dr]^{\pr_G}& \\
H \ar@{..>}[rr] && G.
}
\]
\end{Rec}

\begin{Rem}
In the case of a group~$G$, the Grothendieck construction is often denoted $G\ltimes S$ and called  \emph{transport groupoid} or \emph{action groupoid} (\cf \Cref{Rem:crossed-prod-inv}).
\end{Rem}

\begin{Lem} \label{Lem:Groth-constr-fun}
The construction in \Cref{Rec:Groth-constr} defines a functor
\[
\textstyle\int := \textstyle\int_G^H \colon \biset(H,G) \too \Span(H,G)
\]
for every pair of groupoids $H,G$, by mapping a natural transformation  $\varphi\colon S\Rightarrow T$ of bisets $S,T\colon H^\op\times G\to \set$ to the morphism of spans
\[
[\textstyle{\int}\varphi,\id,\id]=\quad\vcenter{
\xymatrix@R=5pt{
&{\int S} \ar[dd]_{\int \varphi} \ar[dl]_{\pr_H} \ar[dr]^{\pr_G}& \\
H  && G \\
& {\int T} \ar[ul]^{\pr_H} \ar[ur]_{\pr_G} &
}}
\]
where both triangles commute and the functor $\int \varphi \colon \int S\to \int T$ sends an object $(y,x,s)$ to $(y,x,\varphi(s))$ and a map $(h,g)$ to $(h,g)$.
\end{Lem}

\begin{proof}
Straightforward verification.
\end{proof}

\begin{Lem} \label{Lem:ess-image-int}
A span $H \xleftarrow{u} P \xrightarrow{i} G$ belongs to the essential image of $\int_G^H$ if and only if it is \emph{jointly faithful}, \ie the functor $(u,i)\colon P\to H\times G$ is faithful. (For instance, the latter holds if $i\colon P\to G$ is faithful, \ie if the span is right-faithful.)
\end{Lem}

\begin{proof}
First notice that $(\pr_H,\pr_G)\colon \int S\to H\times G$ is (trivially!) jointly faithful for any biset~$S$, and that the property of being jointly faithful is stable under taking isomorphic spans.
Conversely, let $i_!u^*=(H \xleftarrow{u} P \xrightarrow{i} G)$ be any span. There is a canonical morphism of spans as follows
\begin{equation}
\label{eq:aux-Phi}%
[\Phi,\id,\id]=\qquad
\vcenter{\xymatrix@R=5pt{
&P \ar[dd]_{\Phi} \ar[dl]_{u} \ar[dr]^{i}& \\
H  && G \\
& {\int \cat R(i_!u^*)} \ar[ul]^(.6){\pr_H} \ar[ur]_(.6){\pr_G} &
}}
\end{equation}
where the functor
$ \Phi:=  \Phi_{i_!u^*} \colon P\too \textstyle\int \cat R(i_!u^*) $
sends an object
$z$ in~$P$ to the object $(u(z), i(z), [\id_{i(z)} , \id_{u(z)}]) $ in $\int_G^H \big( G(i - , -) \times_P H(-, u-)\big)$, and maps a morphism
$p\in P(z,z')$ to the pair $(u(p),i(p))$; the latter defines a morphism $(uz,iz, [\id, \id])\to (uz',iz', [\id,\id])$ in $\int i_!u^*$, as required, since
\[
i(p) \cdot [\id_{iz} , \id_{uz}] = [i(p), \id_{uz}] = [\id_{iz'} , u(p)] = [\id_{iz'}, \id_{uz'}] \cdot u(p) .
\]
Clearly $\Phi$ is a functor such that $\pr_G\circ \Phi = i$ and $\pr_H\circ \Phi = u$.
It is always full: Given \emph{any} morphism $(h,g)\colon (u(z), i(z), [\id,\id]) \to (u(z'), i(z'), [\id,\id])$ in the target groupoid, that is a $g\in G(iz,iz')$ and an $h\in H(uz,uz')$ such that
\[
g\cdot [ \id_{iz} , \id_{uz}]
\overset{\textrm{def.}}{=} [g , \id_{uz}]
= [\id_{iz'}, h]
\overset{\textrm{def.}}{=} [\id_{iz'}, \id_{uz'}] \cdot h
\]
in $\cat R(i_!u^*)(z,z') =G(i-,iz')\times_P H(uz,u-)$, by definition this means that there exists some $p\in P(z,z')$ such that $\id \circ i(p) = g$ and $u(p) \circ \id = h$, that is:
$(h,g) = (u(p), i(p))= \Phi(p)$.
The functor $\Phi$ is also always essentially surjective: Given \emph{any} object $(y, x, s)$ with $y\in \Obj H$, $x\in \Obj G$ and $s= [g \in G(iz,x), h \in H(y,uz)]$ at some $z\in \Obj Z$, the pair $(h^{-1} , g)$ defines an isomorphism
\[
\Phi(z)= \big( u(z), i(z), [\id_{i(z)} , \id_{u(z)}] \big)
\overset{\sim}{\too}
\big( y , x , [g,h] \big)
\]
because $g \cdot [\id,\id] = [g, h] \cdot h^{-1}$.
Finally, it is easy to see that $\Phi$ is faithful precisely when $(u,i)\colon P\to H\times G$ is faithful.

In short, $i_!u^*$ is jointly faithful if and only if $\Phi$ is an equivalence of groupoids, if and only if the morphism $[\Phi,\id,\id]$ in~\eqref{eq:aux-Phi} defines an isomorphism of spans $i_!u^*\xrightarrow{\sim} \int \cat R(i_!u^*)$, by \cite[Lem.\,5.1.12]{BalmerDellAmbrogio20}. The statement follows.
%In particular, this proves the claimed characterization of the essential image of~$\int$.
\end{proof}

\begin{Lem} \label{Lem:S-vs-Rint(S)}
There is a canonical isomorphism $\varphi_S\colon S \overset{\sim}{\Rightarrow} \cat R\big(\int_G^H S\big)$ for every $G,H$-biset $S\colon H^\op\times G\to \set$.
\end{Lem}

\begin{proof}
Define $\varphi_S$ by setting its component at $(y,x)\in H^\op \times G$ to be the map
\begin{align*}
\varphi_{S,y,x}  \colon  S(y,x) & \too G(\pr_G - , x) \times_{\int S} H (y , \pr_H - )  \\
  s \;\;\; & \mapsto \;\;\; [\id_x , \id_y ] \textrm{ at } (y,x,s) \in \Obj(\textstyle{\int} S)
\end{align*}
Its inverse, say~$\psi_S$, has components given at each $(y,x)$ as follows:
\begin{align*}
\psi_{S,y,x} \colon  G(\pr_G - , x) \times_{\int S} H (y , \pr_H - ) & \too S(y,x) \\
[u_1 \in G(x_1,x),v_1 \in H(y,y_1)] \textrm{ at } (y_1,x_1,s_1) \;\; & \mapsto \;\; u_1\cdot s_1 \cdot v_1
\end{align*}
We leave to the reader the straightforward verifications that $\varphi_S$ and $\psi_S$ are well-defined, mutually inverse natural transformations.
\end{proof}

\begin{Prop}
\label{Prop:local-equivs}
For every pair $H,G$ of finite groupoids, the realization pseudo-functor of \Cref{Rec:realization} and the Grothendieck construction of \Cref{Lem:Groth-constr-fun} restrict to an equivalence of Hom categories
\[
\xymatrix@1{
\Span^\rfree (H,G) \ar@<.8ex>[rr]_-{\sim}^-{\cat R_{H,G}} && \biset^\rfree (H,G) \ar@<.8ex>[ll]^-{\int_G^H}
} .
\]
\end{Prop}

\begin{proof}
Just combine \Cref{Lem:rf}, \Cref{Lem:Groth-constr-fun}, \Cref{Lem:ess-image-int} and \Cref{Lem:S-vs-Rint(S)}.
\end{proof}

\begin{proof}[Proof of \Cref{Thm:Span-vs-biset}]
By \Cref{Lem:rf},  we may restrict $\cat R$ to a pseudo-functor
\[
\cat R \colon \Span^\rfree \too \biset^\rfree
\]
between the 2-full sub-bicategories, restricting to right-faithful in $\Span $ and to right-free in $\biset$.
By \Cref{Prop:local-equivs}, this pseudo-functor is an equivalence at each Hom category. As $\cat R$ is a bijection on objects by construction, we may already conclude that it is a biequivalence of its source and target bicategories.
It also follows that the Grothendieck construction functors of \Cref{Prop:local-equivs} inherit from $\cat R$ through the Hom-equivalences the structure of a pseudo-functor~$\int$, quasi-inverse to~$\cat R$.
\end{proof}

Let us compute the image of some elementary 1-cells and 2-cells under the equivalence~$\cat R\colon \kk\Spanhatrf\isoto \kk\bisethatrf$. These are all easy computations from the definitions.
\begin{Exa}
\label{Exa:R(u^*)}%
Let $u\colon H\to G$ be a group homomorphism and consider the 1-cell $u^*\colon G \xlto{u} H \xto{\Id}H$ in~$\Spanhatrf(G,H)$. Then $\cat R(u^*)$ is the $H,G$-biset~${}_HG_G$ with action $h\cdot x\cdot g=u(h)\,x\,g$.
\end{Exa}
\begin{Exa}
\label{Exa:R(i_!)}%
Let $i\colon H\into G$ be an injective group homomorphism and consider $i_!=i_*\colon H \xlto{\Id} H \xto{i}G$ in~$\Spanhatrf(H,G)$. Then $\cat R(i_!)$ is the $G,H$-biset~${}_GG_H$ with action $g\cdot x\cdot h=g\,x\,i(h)$. In that case, we can combine this 1-cell with the 1-cell~${}_HG_G$ of \Cref{Exa:R(u^*)} and consider the units and counits of the adjunctions~$i_!\adj i^*\adj i_*$. Their images under~$\cat R$ are as follows. Note that ${}_HG_G \circ {}_GG_H\cong {}_HG_H$ whereas ${}_GG_H \circ {}_HG_G\cong {}_G(G\times_H G)_G$. Of course, $\Id_{H}={}_H H_H$ and $\Id_{G}={}_G G_G$. We have
\begin{equation}
\label{eq:(co)units}%
\vcenter{\xymatrix@C=0em@R=.5em{
\leta_i= \big[ H \overset{\id}{\Leftarrow} H \overset{i}{\To}  G \big] \colon
& \Id_H={}_H H_H\To {}_H G_H=i^*i_!
\\
\leps_i=\big[ G\times_H G \overset{\id}{\Leftarrow} G\times_H G \overset{\mu}{\To} G \big] \colon
& i_!i^*= {}_G (G\times_H G)_G\To{}_G G_G= \Id_G
\\
\reta_i=\big[ G\overset{\mu\,}{\Leftarrow} G\times_H G \overset{\id}{\To} G\times_H G \big] \colon
& \Id_G={}_G G_G\To {}_G G\times_H G_G=i_*i^*
\\
\reps_i=\big[ G \overset{i}{\Leftarrow} H \overset{\id}{\To} H\big] \colon
& i^*i_*={}_H G_H \To {}_H H_H=\Id_H
}}\kern-1.3em
\end{equation}
where the map marked~$\mu$ is the multiplication of~$G$.
We recognize the Frobenius relation $\reps\circ\leta=\big[ H=H=H \big]=\id_{H}$.
On the other hand, the composite $\leps\circ\reta$ of \Cref{Def:co2mack-intro} is given by the span of~$G,G$-bisets~$\big[ G\overset{\mu\,}{\Leftarrow} G\times_H G \overset{\mu}{\To} G \big]$.
\end{Exa}

\begin{Rem}
In view of \Cref{Lem:ess-image-int}, it is tempting to believe that $\cat{R}$ should give a biequivalence between $\biset$ and the 2-full subcategory of $\Span$ of all jointly faithful spans.
Unfortunately the latter spans are not stable under composition, hence such a bicategory does not exist.
To see why, just observe the diagram
\begin{equation*}
\vcenter{
\xymatrix@C10pt@R5pt{
&& G \ar@{=}[ld] \ar@{=}[rd] \ar@{}[dd]|-{=} \ar@/_3ex/[ddll] \ar@/^3ex/[ddrr]
\\
&G \ar@{=}[rd] \ar[ld]
&& G \ar@{=}[ld] \ar[rd]
\\
1 \ar@{..>}[rr]
&& G \ar@{..>}[rr]
&& 1
}}
\end{equation*}
displaying a (very) non-jointly-faithful composite of two jointly faithful spans.
\end{Rem}

\begin{Rem}
\label{Rem:crossed-prod-inv}%
For  $H=1$ the trivial group, \Cref{Prop:local-equivs} yields an equivalence
\[
\xymatrix@1{
\tau_1(\gpdG)\cong \Span^\rfree (1,G) \ar@<.8ex>[rr]_-{\sim}^-{\cat R_{1,G}} && \biset^\rfree (1,G) \cong G\sset \,. \ar@<.8ex>[ll]^-{\int_G^1}
}
\]
The two identifications, at the left with the truncated comma 2-category $\tau_1(\gpdG)$ of groupoids faithfully embedded in~$G$, and on the right with the category of left $G$-sets, are isomorphisms of 1-categories simply obtained by suppressing the data over the trivial group.
When $G$ is a group, this is the equivalence of categories used in \cite[App.\,B]{BalmerDellAmbrogio20} and \cite{DellAmbrogio22a} to reformulate Mackey functors for a fixed group~$G$ in terms of groupoids.
Indeed, the above equivalence $G\sset \overset{\sim}{\to} \tau_1(\gpdG)$ is precisely the crossed product functor $X\mapsto (\pi_X\colon G\ltimes X\rightarrowtail G)$ of \cite[Prop.\,B.0.8]{BalmerDellAmbrogio20}, for which we also have now (even for $G$ any finite groupoid) a nice canonical pseudo-inverse $ \tau_1(\gpdG) \overset{\sim}{\to} G\sset $.
Explicitly, the latter sends an object $(P,i_P\colon P\rightarrowtail G)$ to the $G$-set $(i_P/-)_\simeq $ which maps each object $x\in \Obj G$ to the set $(i_P/x)_{\simeq}$ of isomorphism classes of objects $(y,i_P(y)\to x)$ in the ordinary slice category~$(i_P/x)$.
\end{Rem}

\begin{Rem}
\label{Rem:Webbouc}%
By 1-truncating the biequivalence $\cat{R}\colon \Motk \overset{\sim}{\to} (\kk \bisethatrf)^\flat$ and forming categories of $\kk$-linear functors on both sides (see \Cref{Rem:bisets-agree} and \cite[Cor.\,6.22]{DellAmbrogio22a}), we obtain the well-known equivalence between inflation functors~\cite{Webb93} and right-free biset functors~\cite{Bouc10}.
\end{Rem}

%------------------------------------------------------------------------------
%------------------------------------------------------------------------------
%
\section{Cohomological Mackey 2-functors}
\label{sec:cohom-2-Mackey}%
%\bigbreak
%------------------------------------------------------------------------------

Recall from \Cref{Def:co2mack-intro} that a (rectified) Mackey 2-functor is \emph{cohomological} if $\leps \,\reta = [G\!:\!H]\, \id_{\MM(G)}$ for every subgroup inclusion $i\colon H\into G$.
In this section, we provide some familiar examples as well as the first applications of our definition.

\begin{Exa}[Representation theory]
\label{Exa:linear}
There are Mackey 2-functors whose value $\MM(G)$ at a group $G$ is either the category of linear representations $\Modcat(\kk G)$, or its derived category $\Der(\kk G)$, or its stable module category $\Stab(\kk G)$ (for the third example, one limits the domain to $\GG=\gpdf$; see \cite[Ch.\,4.1-2]{BalmerDellAmbrogio20} for details).
An easy computation with the usual adjunctions shows they are all cohomological.
\end{Exa}

\begin{Exa} [Permutation modules]
\label{Exa:perm}
The full subcategories $\permcat_\kk(G) \subseteq \Modcat(\kk G)$ of finitely generated permutation $\kk G$-modules form a Mackey sub-2-functor, since inductions and restrictions preserve permutation modules. It is still cohomological.
The same holds for the subcategory $\permcat_\kk( G)^\natural \subseteq \Modcat(\kk G) $ of direct summands of permutation modules.
(In characteristic~$p$, these are known as \emph{$p$-permutation} modules or \emph{trivial source} modules. But the above makes sense for any ring~$\kk$.)
\end{Exa}

\begin{Exa} [Represented 2-functors]
\label{Exa:univ}
Once we develop our motivic theory, we will see that every cohomological Mackey 2-motive represents a cohomological Mackey 2-functor (\Cref{Cor:univ}).
The trivial motive~$(1,\id)$ represents the Mackey 2-functor $\permcat_\kk( G)^\natural$ of \Cref{Exa:perm}.
\end{Exa}

\begin{Exa} [Equivariant objects]
\label{Exa:geom}
All $G$-local examples of Mackey 2-functors of equivariant objects from \cite[Ch.\,4.4]{BalmerDellAmbrogio20} are cohomological, by an easy computation using the concrete adjunctions provided there. (Here as before, `$G$-local' means defined over $\gpdG\cong G\sset$ rather than~$\gpd$; \cf \Cref{Exas:ordinary-MF}.) These include many geometric examples, such as equivariant coherent sheaves on a noetherian $G$-scheme.
\end{Exa}

\begin{Exa} [Cohomological Mackey functors]
\label{Exa:CoMack}
There is a cohomological Mackey 2-functor whose value
$\MM(G) = \CoMack_\kk(G)$ is
the category of (ordinary) cohomological Mackey functors for~$G$.
To see this, we can apply the construction $\mathcal S \mapsto \MM $ of \cite[Prop.\,7.3.2]{BalmerDellAmbrogio20}, as we can easily check that if the Mackey 2-functor $\mathcal S$ is cohomological then so is~$\MM$.
Taking $\mathcal S$ the (pointwise dual of) the Mackey 2-functor $ G\mapsto \permcat_\kk(G)^\op$ of~\Cref{Exa:perm}, we obtain $\MM\colon G\mapsto \Fun_\kk (\permcat_\kk( G), \Modcat \kk)\cong \CoMack_\kk(G)$ in accordance with Yoshida's Theorem (\cf \Cref{Rem:cf-Yoshida}).
\end{Exa}

The next result is a first justification for the adjective `cohomological'.

\begin{Thm} \label{Thm:coh-decats}
Let $\MM$ be any cohomological Mackey 2-functor, and suppose that $M$ is an ordinary Mackey functor obtained from~$\MM$ by the Hom-decategorification procedure as in \Cref{Thm:Hom-decat-general}.
Then $M$ is cohomological in the classical sense: $I_K^H R^H_K = [H\!:\!K]\cdot \id_{M(H)}$ for all subgroups $K\le H\leq G$.
\end{Thm}

\begin{proof}
This is a direct verification from the definition of the restrictions and transfers in \Cref{Thm:Hom-decat-general}, in fact for any of the classical choices of $(\GG;\JJ)$ as in \Cref{Exas:ordinary-MF} where we can  view subgroup inclusions as 1-morphisms in~$\JJ$.

Namely, let $(i\colon H\to G)\in \JJ$ and $f\in M(G)= \MM(G)(X_G, Y_G)$ for any coherent choice $\{X_G,Y_G,\lambda_u,\rho_u\}$ of pairs of objects  in~$\MM$.
Applying the transfer $i_\sbull$ to the restricted map $i^\sbull(f)$, the $\lambda_i$'s and $\rho_i$'s cancel out, leaving us with the composite
\[
\xymatrix{
X_G \ar[r]^-{\reta_{i}}
& i_*i^* X_G \ar[r]^-{i_*i^*(f)}
& i_*i^*Y_G \ar[r]^-{\leps_i}
& Y_G.
}
\]
By naturality of~$\reta_{i}$ or~$\leps_{i}$, this is $f$ composed with $\leps_{i}\circ\reta_{i}=[H\!:\!K]\,\id$.
\end{proof}

\begin{Exa}[Group cohomology]
Specializing \Cref{Exa:general-monoidal} to the (global or local) cohomological Mackey functor $G\mapsto \MM(G) = \Der(\kk G)$ of \Cref{Exa:linear} yields the motivating example of a cohomological Mackey functor, namely group cohomology
$G\mapsto \mathrm H^* (G;\kk) = \MM(G)^*(\mathbf 1, \mathbf 1)$.
For a fixed group $G$, and any $V\in \Modcat (\kk G)$, we also get the variant $H \mapsto \mathrm H^*(H; V|_H)$ ($H\leq G$) with twisted coefficients by setting $X_0 = \kk$ and $Y_0 = V$ in \Cref{Exa:local-general}.
Similarly, by taking stable module categories instead of derived categories we obtain Tate cohomology of finite groups.
\end{Exa}

\begin{Rem}
\label{Rem:decats}
The analogue of \Cref{Thm:coh-decats} does \emph{not} hold for the perhaps more familiar $K_0$-style of decategorification (see \cite[\S 2.5]{BalmerDellAmbrogio20}).
For instance, the Mackey 2-functor $G\mapsto \modcat(\kk G)$ of finitely generated representations is cohomological (\cf \Cref{Exa:linear}) but the ordinary Mackey functor $G\mapsto K_0(\modcat \kk G )=\mathrm{R}_\kk(G)$, say for $\kk=\mathbb C$, is the usual representation ring, which is not cohomological.
\end{Rem}

\begin{Rem}
\label{Rem:comack-gen-formula}
Following up on the previous remark, suppose that $\MM$ is a \emph{Green 2-functor} in the sense of \cite{DellAmbrogio21pp}.
In particular, this means that the categories $\MM(G)$ are monoidal, the restriction functors are strong monoidal, and there is a projection formula
$i_*(i^*X \otimes Y)\cong X \otimes i_*Y$ for faithful $i\colon H\into G$.
After applying~$K_0$, one gets an ordinary Mackey functor satisfying $I_H^G R^G_H=[i_*(\mathbf 1)]\cdot\id$, where $\mathbf 1\in \MM(H)$ denotes the tensor unit.
We can view this as a kind of `generalized cohomological relation' for a `generalized index'~$[i_*(\mathbf 1)]\in K_0(\MM(G))$.
This applies to the representation ring (\Cref{Rem:decats}), and in fact to any Green functor arising from a Green 2-functor by $K_0$-decategorification.
If $\MM$ is a Green 2-functor, moreover, the projection formula $i_*i^*\cong i_*(\mathbf{1})\otimes -$ identifies the composite $\leps\circ\reta$ of \Cref{Def:co2mack-intro} with multiplication by the Euler characteristic $\chi(i_*(\mathbf{1})) \in \End_{\MM(G)}(\mathbf{1})$ (\ie the monoidal trace of the identity) of the dualizable object~$i_*(\mathbf{1})$.
This is a consequence of the special Frobenius structure of $i_*(\mathbf{1})$; see \cite[\S\S\,7-8]{DellAmbrogio21pp}.
Hence such an $\MM$ is cohomological precisely when $\chi(i_*(\mathbf{1}))$ is multiplication by $[G\!:\!H]$.
\end{Rem}

Thus cohomological Mackey 2-functors are a source of classical cohomological Mackey 1-functors, via Hom-decategorification.
For the remainder of the section, we further validate our definition by sketching a couple of applications.

\begin{Thm}[$p$-local separable monadicity]
\label{Thm:p-local-rec}
Let $\MM$ be a cohomological Mackey 2-functor, and let $G$ be a finite group such that $\MM(G)$ is a $\mathbb Z_{(p)}$-linear category for a prime number~$p$ (\eg the base ring $\kk$ is a $\bbZ_{(p)}$-algebra, \eg it is a field of characteristic~$p$).
Let $i\colon H\into G$ denote the inclusion of a subgroup of index prime to~$p$ (for instance a $p$-Sylow).
Then the monad $\mathbb A := i^*i_!$ on $\MM(H)$ induced by the adjunction $i_!\dashv i^*$ is \emph{separable}, that is its multiplication $\mu:=i^* (\leps ) i_*\colon \mathbb A^2\Rightarrow \mathbb A$ admits an $\mathbb A$-bilinear section. In particular, it follows that restriction $i^*\colon \MM(G)\to \MM(H)$ satisfies descent, in that the canonical comparison functor $\MM(G) \overset{\sim}{\too} \Modcat(\mathbb A)_{\MM(H)}$ into the Eilenberg-Moore category of $\mathbb A$-modules in~$\MM(H)$ is an equivalence.
\end{Thm}

\begin{proof}
Since $\MM$ is cohomological, the composite $\leps \,\reta$ acts on  $\MM(G)$ as multiplication by $[G:H]$, which is invertible by the $\mathbb Z_{(p)}$-linearity of~$\MM(G)$. In particular, the counit $\leps$ of the adjunction $i_!\dashv i^*$ admits a natural section.
As $\MM(G)$ is assumed idempotent-complete, we may conclude with \cite[Lemma~2.10]{Balmer15}.
%(the resulting section of $\mu$ is then $\frac{1}{[G:H]} i^* (\reta) i_*\colon \mathbb A\Rightarrow \mathbb A^2$).
\end{proof}

\begin{Rem} Note that \Cref{Thm:p-local-rec} goes in the `opposite' direction of the deceptively similar result of \cite[Theorem~2.4.1]{BalmerDellAmbrogio20}, which says that the \emph{other} adjunction $i^*\dashv i_*$ is separably monadic, and which holds for \emph{any} Mackey 2-functor $\MM$ and \emph{any} faithful $i\colon H\rightarrowtail G$. In particular, one can always reconstruct $\MM(H)$ from the monad~$i_*i^*$ on~$\MM(G)$, but in order to recover $\MM(G)$ from the monad $i^*i_!$ on $\MM(H)$ special circumstances are required, such as those in \Cref{Thm:p-local-rec}.
\end{Rem}

\begin{Rem}
\Cref{Thm:p-local-rec} can be reformulated as a categorification of the classical Cartan-Eilenberg stable elements formula for mod-$p$ group cohomology.
To wit, for $\MM$ and $\kk$ as in the theorem, the above monadic reconstruction of $\MM(G)$ can be replaced by a pseudo-limit $\MM(G)\simeq \lim _{G/P} \MM(P)$ in $\ADD$, where $P$ ranges through the $p$-subgroups of~$G$ and the limit is taken over the corresponding orbit category.
This leads to a generalization of the main results of \cite{Balmer15} to arbitrary cohomological Mackey 2-functors;
see \cite{Maillard22} for details and explanations.
\end{Rem}

\begin{Rem}
To outline another application, let us simply mention that the general Green correspondence of \cite{BalmerDellAmbrogio21} is most useful for Mackey 2-functors which are cohomological, for which it gives rise to a `$p$-local theory' as in modular representation theory. See \cite[\S\S\,6-7]{BalmerDellAmbrogio21} for details.
\end{Rem}

%------------------------------------------------------------------------------
%
\section{Cohomological Mackey 2-motives}
\label{sec:cohom-2-motives}%
%\bigbreak
%------------------------------------------------------------------------------

We now turn to \Cref{Thm:univ-cohom-2-Mackey-intro} and the description of cohomological Mackey 2-motives in simple terms.
Here we use the results of \Cref{sec:2-Mack-bisets}.

Since cohomological Mackey 2-functors are Mackey 2-functors that send some special 2-cells to zero, there is a tautological approach to the bicategory of cohomological Mackey 2-motives: It is the quotient
\[
\cat Q\colon\Motk\onto \CohMotk
\]
obtained by modding out the 2-cells corresponding to the cohomological conditions:
\begin{equation}
\label{eq:coh-rel}%
\qquad \leps_i \circ \reta_i  - [G\!:\!H]\cdot \id \qquad \in \End_{\Motk}(\Id_{G})
\end{equation}
for every inclusion $i\colon H\into G$ of a subgroup~$H$ in a group~$G$. But this definition is rather sprawling: We need to consider the closure of the above class of 2-cells inside~$\Motk$ under composition and $\kk$-linear combination inside each Hom category; plus we need to take into account horizontal composition with arbitrary 2-cells, including whiskering. So the tautological definition is unwieldy.

Our \Cref{Thm:univ-cohom-2-Mackey-intro} gives a concrete realization of~$\CohMotk$ as~$(\bipermk)^\flat$, the bicategory obtained by block-completing the bicategory of right-free permutation bimodules~$\biperm_{\kk}^\rfree$ (\Cref{Def:biperm-rfree}).
In view of \Cref{Prop:k-linearization}, it is more convenient to use the model of $\Motk$ via bisets, as described in \Cref{sec:2-Mack-bisets}. To do so, we need to translate the 2-cell~\eqref{eq:coh-rel} under the equivalence of \Cref{Thm:Span-vs-biset}.
In view of \Cref{Exa:R(i_!)}, the image of $\leps_i \circ \reta_i-[G\!:\!H]\cdot \id$ in the category $\kk\bisethatrf(G,G)$ is simply the following linear combination of spans of equivariant maps between right-free $G,G$-bisets (with $\mu$ induced by multiplication):
\begin{equation}
\label{eq:coh-rel-biset}%
\big[ G\overset{\mu\,}{\Leftarrow} G\times_H G \overset{\mu}{\To} G \big] - [G\!:\!H]\cdot \id_{\Id_G}.
\end{equation}
For simplicity, we call this the \emph{cohomological 2-cell} corresponding to~$H\le G$.

\begin{Rem}
\label{Rem:bisets-x}%
Let $G_1,G_2$ be two groups. We can view $G_1,G_2$-bisets~$X$ as left $(G_1\times G_2)$-sets via $(g_1,g_2)\cdot x=g_1 x g_2\inv$ for every $x\in X$. Decomposing into orbits, every $G_1,G_2$-biset is a coproduct of transitive $G_1\times G_2$-sets of the form $(G_1\times G_2)/M$ for subgroups~$M\le G_1\times G_2$.
Translating back, the $G_1,G_2$-biset corresponding to such an orbit $(G_1\times G_2)/M$ is the same set with action $g_1\cdot[x,y]\cdot g_2=[g_1\,x,g_2\inv y]$.
Since we focus on right-free bisets, we note that $(G_1\times G_2)/M$ is right-free (over~$G_2$) if and only if the first projection $G_1\times G_2\to G_1$ is injective on~$M$, that is, $(\pr_1)\restr{M}\colon M\into G_1\times G_2\onto G_1$ is faithful. Indeed, the right-action of $g_2\in G_2$ fixes a class $[x,y]\in(G_1\times G_2)/M$ if and only if $(1,(g_2)^y)\in M$.
\end{Rem}
\begin{Exa}
\label{Exa:Delta-1}%
Let $G_1=G_2=G$ and consider the $G,G$-biset $G\times_H G$ of~\eqref{eq:coh-rel-biset}. As left $(G\times G)$-set, it is isomorphic to a single orbit $(G\times G)/M$ for the subgroup $M=\Delta(H)=\SET{(h,h)}{h\in H}$, via $(G\times G)/M\isoto G\times_H G$, $(g_1,g_2)\mapsto (g_1,g_2\inv)$.
\end{Exa}

We now consider a class of 2-cells in~$\kk\bisethatrf$ that will play a role later on.
\begin{Cons}
\label{Cons:delta}%
Let $G_1,G_2$ be finite groups and~$M\le N\le G_1\times G_2$ be subgroups such that $\pr_1$ is injective on~$N$.
As in \Cref{Rem:bisets-x}, we view the linear combination
\begin{equation}
\label{eq:coh-rel-biset-delta}%
\delta(G_1,G_2,M,N):=\quad \Big[ \textstyle{\frac{G_1\times G_2}{N}} \Leftarrow \textstyle{\frac{G_1\times G_2}{M}} \Rightarrow \textstyle{\frac{G_1\times G_2}{N}} \Big] \ - \ [N\colon M]\cdot\id_{(G_1\times G_2)/N}
\end{equation}
as an endomorphism of the $G_1,G_2$-biset $(G_1\times G_2)/N$ in the category $\kk\bisethatrf(G_2,G_1)$. The two equivariant maps denoted `$\Rightarrow$' are simply the quotient maps.
\end{Cons}
\begin{Exa}
\label{Exa:Delta-2}%
Let $H\le G$. Take again $G_1=G_2=G$ as in \Cref{Exa:Delta-1}.
Taking the subgroups~$M=\Delta(H)$ and~$N=\Delta(G)$ in \Cref{Cons:delta}, the expression~\eqref{eq:coh-rel-biset-delta} boils down to our cohomological 2-cell~\eqref{eq:coh-rel-biset}. Conversely, we now prove that every $\delta(G_1,G_2,M,N)$ belongs to the ideal generated by the cohomological 2-cells.
\end{Exa}

\begin{Lem}
\label{Lem:coh-rel-delta}%
Let $G_1$ and $G_2$ be finite groups and $M\leq N \leq G_1 \times G_2$ subgroups such that the first projection $\pr_1\colon G_1\times G_2\to G_1$ is injective on~$N$ (and thus on~$M$). Then the 2-cell $\delta(G_1,G_2,M,N)$ given in~\eqref{eq:coh-rel-biset-delta} belongs to the (bicategorical) ideal of 2-cells in~$\kk\bisethatrf$ generated by the cohomological 2-cells~\eqref{eq:coh-rel-biset}.
\end{Lem}
\begin{proof}
Consider the cohomological 2-cell~\eqref{eq:coh-rel-biset} for the subgroup~$H:=M$ of~$G:=N$, that is, $\delta_0:=\big[N\Leftarrow N\times_M N \To N\big]-[N\!:\!M]\cdot\id_{\Id_N}$. We claim that $\delta(G_1,G_2,M,N)$ is simply the 2-cell obtained from (pre-)whiskering $\delta_0$ by the 1-cell $G_2\to N$ given by the $N,G_2$-biset $G_2$ and (post-)whiskering it by the 1-cell $N\to G_1$ given by the $G_1,N$-biset $G_1$. In both cases, $N$ acts on~$G_i$ via $(\pr_i)\restr{N}\colon N\into G_1\times G_2\onto G_i$. Note that the resulting 1-cell is as wanted:
\[
G_1\times_N N \times_N G_2=G_1\times_N G_2\leftrightarrow(G_1\times G_2)/N
\]
where $\leftrightarrow$ indicates the dictionary between $G_1$,$G_2$-bisets and left $(G_1\times G_2)$-sets of \Cref{Rem:bisets-x}.
Since whiskering $\id_{\Id_N}$ in the same way gives $\id_{(G_1\times G_2)/N}$, it suffices to see what happens to the span $\big[ N\Leftarrow N\times_M N \To N \big]$ under these whiskerings. The result is indeed
\[
\big[\; (G_1\times G_2)/N\Longleftarrow (G_1\times G_2)/M \Longrightarrow (G_1\times G_2)/N \;\big]
\]
since the $G_1,G_2$-biset $G_1\times_N (N\times_M N)\times_N G_2\cong G_1\times_M G_2$ translates into the orbit~$(G_1\times G_2)/M$ as a left $(G_1\times G_2)$-set, via \Cref{Rem:bisets-x}. Direct verification shows that the above maps $\Leftarrow$ and $\To$ are indeed the canonical projections.
\end{proof}

In \Cref{Def:biperm-rfree}, we encountered the $\kk$-linearization $\kk[-]\colon \biset^\rfree\to \biperm^\rfree_\kk$.
Interestingly, $\biperm^\rfree_\kk$ accommodates spans, that is $\kk[-]$ can be nicely extended along the inclusion $\biset^\rfree \subset \bisethatrf$ to a pseudo-functor defined on \emph{spans} of 2-cells:

\begin{Prop}
\label{Prop:P}
There is a well-defined pseudo-functor
\[
\cat P\colon \bisethatrf\too \bipermk
\]
which maps a finite groupoid to itself and a $G,H$-biset $U$ to the permutation $G$,$H$-bimodule~$\kk[U]$. It maps a 2-cell given by a span of equivariant maps of $G,H$-bisets
\[
\big[ \xymatrix{
U & W \ar@{=>}[l]_-{\beta} \ar@{=>}[r]^-\alpha &  V
}
\big]
\]
to the \emph{sum-over-fibers} natural transformation $\alpha_\star\beta^\star \colon \kk[U]\Rightarrow \kk[V]$, whose component at the object $(y,x)\in H^\op\times G$ is the $\kk$-linear map defined on basis elements by
\begin{equation} \label{eq:sums-over-fibers}
(\alpha_\star\beta^\star)_{y,x} \colon \kk[U(y,x)] \too \kk[V(y,x)], \quad u\longmapsto \sum_{w\in \beta_{y,x}^{-1}(u)} \alpha_{y,x}(w)\,.
\end{equation}
\end{Prop}

\begin{proof}
This is a lengthy verification that we only outline.
Local functoriality of $\cat P\colon \bisethatrf(H,G)\to \bipermk(H,G)$ entails that given a pullback of $G,H$-bisets
\[
\xymatrix@C=1.5em@R=1.5em{
& U \ar@{=>}[ld]_-{\alpha} \ar@{=>}[rd]^-{\beta}
\\
V \ar@{=>}[rd]_-{\gamma}
&& W \ar@{=>}[ld]^-{\delta}
\\
& X
}
\]
we have $\beta_\star\alpha^\star=\delta^\star\gamma_\star\colon \kk[V]\to \kk[W]$. This is a direct verification on the bases, by definition of the cartesian product: $\beta$ restricts to a bijection $\alpha\inv(v)\isoto \delta\inv(\gamma(v))$.

To show that $\cat P$ preserves horizontal composition, consider for every $G,H$-biset~$U$ and $H,K$-biset~$V$ the canonical isomorphism $\kk[U\times_H V]\isoto \kk[U]\otimes_{H}\kk[V]$ that we already used in the proof of \Cref{Prop:k-linearization}. It provides the compatibility isomorphism between $\cat P(U\circ V)$ and $\cat P(U)\circ\cat P(V)$, on the condition that it is also natural with respect to the `backwards' morphisms~$U\Leftarrow U':\beta$ and the associated~$\beta^\star $. This is again a direct verification on the bases.
\end{proof}
\begin{Rem} \label{Rem:P-adj-groups}
When $i\colon H\into G$ is the inclusion of a subgroup, we spelled out the adjunctions $i_!\adj i^*\adj i_*$ in~$\bisethatrf$ in \Cref{Exa:R(i_!)}.
These adjunctions now have an image in~$\bipermk$ under~$\cat{P}$.
Of course, $\cat{P}(i_!)=\cat{P}(i_*)=\kk[{}_G G_H]={}_G \kk G_H$ and $\cat{P}(i^*)=\kk[{}_H G_G]={}_H \kk G_G$.
The units and counits $\leta,\leps$ of ${}_G \kk G_H \dashv {}_H \kk G_G$ and those $\reta,\reps$ of ${}_H \kk G_G \dashv {}_G \kk G_H$ are given in $\perm_\kk^\rfree$ by the familiar formulas:
\begin{equation*} %\label{eq:expl-adj}
\vcenter{
\xymatrix@R=6pt@C=6pt{
1 \ar@{|->}[rr] && { \sum_{[y]\in G/H} y\otimes y^{-1}} \\
{}_G\kk G_G \ar@<.5ex>[rr]^-{\reta} && {}_G \kk G \otimes_H \kk G_G \ar@<.5ex>[ll]^-{\leps} \\
g'g && \ar@{|->}[ll] g' \otimes g
}}
\quad\quad\quad
\vcenter{
\xymatrix@R=6pt@C=6pt{
\quad \quad \quad \quad \quad \quad \quad g \ar@{|->}[rr] && {\left\{\begin{array}{ll} g & (g\in H) \\ 0 & (g\not\in H) \end{array}\right. } \\
{}_H\kk G \otimes_G \kk G_H \cong {}_HG_H
  \ar@<.5ex>[rr]^-{\reps} &&
   {}_H\kk H_H \ar@<.5ex>[ll]^-{\leta\,=\,\incl } \\
&&
}}
\end{equation*}
Indeed, it is very easy to follow the images of the spans~\eqref{eq:(co)units} under the `sum-over-fibers' recipe of \Cref{Prop:P}. For instance, $\reps$ is given in~$\bisethatrf$ by the span $\big[ G \overset{i}{\Leftarrow} H \overset{\id}{\To} H\big]$. Its image under~$\cat{P}$ is the morphism $\kk G\to \kk H$ mapping a basis element $g\in G$ to $\sum_{h\in i\inv(g)}h$ which is $g$ if $g\in H$ and zero otherwise.
\end{Rem}
\begin{Exa}
\label{Exa:P(coh-rel)}%
Let us compute the image under~$\cat P$ of the cohomological 2-cell~\eqref{eq:coh-rel-biset}, for $H\le G$. Clearly, the underlying 1-cell~$\Id_{G}$ goes to~$\Id_{G}$ which is the bimodule~$\kk G$. Also clearly $[G\!:\!H]\cdot \id_{\Id_G}$ goes to $[G\!:\!H]\cdot \id_{\kk G}$.
The $\kk$-linear $\cat{P}\big(\big[ G\overset{\mu\,}{\Leftarrow} G\times_H G \overset{\mu}{\To} G \big] \big)\colon\!$ $\kk G\to \kk G$ maps a basis element $g\in G$ to $\sum_{[x,y]\in \mu\inv(g)}xy=|\mu\inv(g)|\cdot g$ where $\mu\colon G\times_HG\to G$ is multiplication. That fiber $\mu\inv(g)$ has~$[G\!:\!H]$ elements. In conclusion, $\cat{P}$ maps the cohomological 2-cell~\eqref{eq:coh-rel-biset} to zero.
\end{Exa}

We are now ready to describe $\bipermk$ as a 2-quotient of~$\kk\bisethatrf$.
\begin{Thm}
\label{Thm:2-Yoshida}%
Consider the unique $\kk$-linear extension $\kk \bisethatrf \to \perm^\rfree_\kk$ of the pseudo-functor $\cat P$ of \Cref{Prop:P}, which we again denote by~$\cat P$.
It is the identity on objects, essentially surjective on 1-cells, and full on 2-cells.
Its kernel on 2-cells $\ker(\cat P):= \{ \varphi \in \kk \bisethatrf_2 \mid \cat P(\varphi)=0 \}$ is generated by the \emph{cohomological 2-cells}~\eqref{eq:coh-rel-biset}
for all inclusions $i\colon H\into G$ of a subgroup $H$ of a finite group~$G$.
\end{Thm}

\begin{proof}
The pseudo-functor~$\cat{P}$ is the identity on objects by definition. It is also essentially surjective on 1-cells since permutation modules are $\kk$-linearizations of bisets. The point is the behavior on 2-cells. To formalize the statement, consider the (\emph{$\kk$-linear, additive, bicategorical}) \emph{ideal}~$\cat{J}$ of~$\kk\bisethatrf$ generated by~\eqref{eq:coh-rel-biset}, meaning the smallest class of 2-cells containing those and stable under horizontal composition with arbitrary 2-cells, and such that its restriction to each Hom category is closed under taking linear combinations and (vertical) composites with arbitrary maps. By \Cref{Exa:P(coh-rel)}, we know that $\cat{P}(\cat{J})=0$. So we have a factorization
\[
\xymatrix@!@C=2pt@R=2pt{
& {\kk \bisethatrf} \ar[dr]^-{\cat P} \ar[dl]_{\textrm{quot.}} &  \\
\cat{B}:=\displaystyle{ \frac{\kk \bisethatrf}{\cat{J}} }
 \ar[rr]^-{\bP} && {\perm^\rfree_\kk}
}
\]
where the left-hand quotient bicategory~$\cat{B}$ has the same objects and 1-cells as $\kk \bisethatrf$ and the obvious quotient by~$\cat{J}$ as 2-cells. The claim of the theorem is that $\bP$ is a biequivalence. Since $\bP$ is, like~$\cat{P}$, the identity on objects and essentially surjective on 1-cells, the claim is that $\bP$ is fully faithful on each Hom category.

Using additivity, we easily reduce to connected groupoids as objects and transitive bisets as 1-cells. So we need to prove the following. Let $G_1$ and $G_2$ be finite groups, $U$ and~$V$ two transitive $G_1,G_2$-bisets, then the $\kk$-linear homomorphism
\begin{equation}
\label{eq:P-bar}%
\bP\colon \Hom_{\cat{B}(G_1,G_2)}(U,V)\to \Hom_{\bipermk(G_1,G_2)}(\kk[U],\kk[V])
\end{equation}
is bijective.
To prove this, we shall find a certain number~$n$ of generators of the left-hand $\kk$-module $\Hom_{\cat{B}(G_1,G_2)}(U,V)$, prove that their images under $\bP$ form a $\kk$-basis of $\Hom_{\bipermk(G_1,G_2)}(\kk[U],\kk[V])$ and prove that the latter is a free $\kk$-module of the same rank~$n$. In particular those generators are $\kk$-linearly independent in $\Hom_{\cat{B}(G_1,G_2)}(U,V)$ already (as their images are) and thus $\bP$ is an isomorphism.

So let us produce those $n$ generators. As in \Cref{Rem:bisets-x}, we view bisets as left $\Gamma$-sets for $\Gamma=G_1\times G_2$ and we are assuming that $U=\Gamma/K$ and $V=\Gamma/L$ for subgroups~$K,L\le \Gamma$. The number $n$ mentioned above will be $n=|K\bs \Gamma/L|$.

By construction of $\cat{B}$ as a quotient of $\kk\bisethatrf$ and by additivity  (for disjoint unions) in the middle object of spans, it is easy to see that $\Homcat{B(G_1,G_2)}(\Gamma/K\,,\,\Gamma/L)$ is generated $\kk$-linearly by equivalence classes of the form
\begin{equation}
\label{eq:aux-general}%
[\; \Gamma / K \overset{\cdot\beta}{\Longleftarrow} \Gamma / M \overset{\cdot\gamma}{\Longrightarrow} \Gamma / L \;]
\end{equation}
for subgroups~$M\le \Gamma$. Both maps in \eqref{eq:aux-general} are necessarily given by (inner) right-multiplication $[x]\mapsto [x\cdot\beta]$ and $[x]\mapsto [x\cdot\gamma]$ by elements $\beta,\gamma\in\Gamma$ such that $M^{\beta}\le K$, respectively $M^{\gamma}\le L$. Replacing $M$ by a conjugate subgroup of~$\Gamma$, we can assume that $\beta=1$. In that case, this span~\eqref{eq:aux-general} is equal to the composite of spans
\[
\xymatrix@R=10pt{
& \Gamma / N \ar@{=>}[dl]_-{\pr} \ar@{=>}[dr]^\id  & & \Gamma / M \ar@{=>}[dl]_-{\pr} \ar@{=>}[dr]^-{\pr} & & \Gamma / N \ar@{=>}[dl]_\id \ar@{=>}[dr]^{\cdot \gamma} & \\
\Gamma / K \ar@{..>}[rr] &  & \Gamma / N \ar@{..>}[rr] & & \Gamma / N \ar@{..>}[rr] & & \Gamma / L
}
\]
where $N$ is short for~$K\cap {}^{\gamma\!}L$ and all `$\pr$' denote canonical projections.
By \Cref{Lem:coh-rel-delta}, the middle span becomes multiplication by the index~$[N\!:\!M]$ in the quotient bicategory~$\cat B$.
Our generator~\eqref{eq:aux-general} is therefore equal to~$[N\!:\!M]$ times the element
\begin{equation}
\label{eq:aux-generators}
[\; \Gamma / K \overset{\pr}{\Longleftarrow} \Gamma / (K \cap {}^{\gamma\!} L) \overset{\cdot \gamma\;}{\Longrightarrow} \Gamma / L \;]
\end{equation}
where $\gamma\in \Gamma$.
Thus the spans~\eqref{eq:aux-generators}, for $\gamma\in \Gamma$, are generators of $\Hom_{\cat{B}(G_1,G_2)}(U,V)$.
Note that the generator \eqref{eq:aux-generators} only depends on the class~$[\gamma]\in K\bs\Gamma/L$, since for every~$k\in K$ and~$\ell\in L$ the following diagram of 2-cells commutes in $\biset^\rfree$:
\[
\xymatrix@R=.1em{
& \Gamma / K\cap {}^\gamma L \ar@{=>}[dl] \ar@{=>}[dr]^{\cdot \gamma} \ar@{=>}[dd]^{\cdot k\inv}_\simeq & \\
\Gamma / K & & \Gamma / L \\
& \Gamma / K\cap {}^{k\gamma\ell} L \ar@{=>}[ul] \ar@{=>}[ur]_(.6){\cdot k\gamma\ell}&
}
\]
So it suffices to take a span~\eqref{eq:aux-generators} for each class~$[\gamma]$ in~$K\bs\Gamma/L$ to obtain our set of $n$ generators of~$\Hom_{\cat{B}(G_1,G_2)}(U,V)$ in~\eqref{eq:P-bar}, where $n=|K\bs\Gamma/L|$.

It only remains to check that the images under $\bP$ of the spans~\eqref{eq:aux-generators} for $[\gamma]\in K\bs\Gamma/L$ form a $\kk$-basis of the $\kk$-module $\Hom_{\bipermk(G_1,G_2)}(\kk[U],\kk[V])$ of~\eqref{eq:P-bar}.
In terms of left $\Gamma$-modules, this $\kk$-module is $\Hom_{\Gamma}(\kk(\Gamma/K),\kk(\Gamma/L))$.
We compute
\begin{align*}
& \Hom_{\Gamma}(\kk[\Gamma/K], \kk[\Gamma/L])
 \,\cong \, \Hom_{\Gamma}(\kk, \kk[\Gamma/K] \otimes_\kk \kk[\Gamma/L])  \, \cong \\
& \,\cong \, \bigoplus_{[\gamma]\in K\bs \Gamma /L} \kern-1em \Hom_{\Gamma} (\kk , \kk [\Gamma / K \cap {}^\gamma L])
 \, \cong \, \bigoplus_{[\gamma]\in K\bs \Gamma /L} \kern-1em \Hom_{K \cap {}^\gamma L} (\kk, \kk)
 \, \cong \, \bigoplus_{[\gamma]\in K\bs \Gamma /L} \kern-1em \kk \ \cong \,\kk^n
\end{align*}
by combining the self-duality of $\kk[\Gamma/K]$ for the tensor product, the Mackey formula, and the adjunction between restriction and induction along $K\cap {}^\gamma L \leq \Gamma$.
By explicitly retracing the element $1\in \kk$ in the summand for $[\gamma]\in K\bs \Gamma/ L$ in the target~$\kk^n$, we find its preimage in $\Hom_{\Gamma}(\kk[\Gamma/K], \kk[\Gamma/L])$ to be the $\kk\Gamma$-linear morphism
$[x]\mapsto \sum_{[y]} [y\gamma]$ with $[y]$ running through those cosets in $\Gamma/K\cap {}^\gamma L$ such that $[y]=[x]$ in~$\Gamma/K$.
The latter map is precisely the image of the generator~\eqref{eq:aux-generators} under~$\bP$, by a direct application of \Cref{Prop:P}. Indeed, the image under~$\cat{P}$ of the span $\big[ \Gamma/K\overset{\pr}\Leftarrow \Gamma / (K \cap {}^{\gamma\!} L) \overset{\cdot \gamma\,\,\;}{\Rightarrow} \Gamma/L \big]$ is by definition $(\cdot\gamma)_\star\pr^\star\colon \kk[U]\to \kk[V]$, that is, it maps a generator $[x]\in U=\Gamma/K$ to the sum $\sum_{[y]\in\pr\inv([x])}[y\gamma]$ over its preimages $[y]\in \Gamma/(K\cap {}^\gamma L)$, which are precisely those~$[y]$ such that $[y]=[x]$ in~$\Gamma/K$.
\end{proof}

\begin{Rem} [{Yoshida's Theorem; see \cite[\S7]{Webb00}}]
\label{Rem:cf-Yoshida}%
Fix a finite group~$G$.
Consider the functor $\cat P_{1,G}\colon \kk \bisethatrf(1,G)\to \perm^\rfree_\kk(1,G)$, the component functor  of our pseudo-functor~$\cat P$ at the pair $(1,G)$.
By \Cref{Rem:crossed-prod-inv}, its source 1-category is
\[
\mathrm{Sp}_\kk(G) := \kk \widehat{(G\sset)} ,
\]
the $\kk$-linear span category of finite left $G$-sets, \ie the classical Burnside category for~$G$.
Its target is just $\permcat_\kk(G)$, the category of finitely generated left permutation $\kk G$-modules. Thus $\cat P_{1,P}$ identifies with `Yoshida's functor'
\[
Y_G \colon \mathrm{Sp}_\kk(G) \to \permcat_\kk(G)
\]
sending a left $G$-set $X$ to $\kk [X]$ and a span of $G$-maps to the associated sum-over-fibers $\kk G$-linear map.
As we know, this functor is $\kk$-linear, essentially surjective and full, and Yoshida's Theorem says that an ordinary Mackey functor for $G$ (\ie an additive functor $\mathrm{Sp}_\kk(G) \to \Modcat(\kk)$) is cohomological if and only if it factors through~$Y_G$. Equivalently, this says that the kernel of $Y_G$ is generated as a $\kk$-linear categorical ideal by the differences
$I^K_LR^K_L - [K:L]\cdot \id_{K/L}$ for all $L\leq K \leq G$.
Thus we can view the results of this section as a categorification of Yoshida's Theorem.
\end{Rem}

We now derive from \Cref{Thm:2-Yoshida} the 2-universal property for $\perm^\rfree_\kk$ and thus a proof of \Cref{Thm:univ-cohom-2-Mackey-intro}.
The moment has also come to define our model for the bicategory of cohomological Mackey 2-motives.

\begin{Def}
 \label{Def:CoMot}
We define the \emph{bicategory of cohomological Mackey 2-motives} to be $\CohMotk:=(\perm^\rfree_\kk)^\flat$, the block-completion of $\perm^\rfree_\kk$ in the sense of  \cite[Constr.\,A.7.22]{BalmerDellAmbrogio20}.
Its objects are pairs $(G,\phi)$ with $G$ a finite groupoid and $\phi$ an idempotent element of the ring $\End_{\kk \perm^\rfree_\kk}(\Id_G)$; a 1-cell $(H,\psi)\to (G,\phi)$ is a pair $(M,\mu)$ with $M$ a right-free permutation $H,G$-module and $\mu= \mu^2$ an idempotent equivariant map $M\Rightarrow M$ absorbing $\phi$ and~$\psi$.
In particular, the Hom category at (the cohomological motives of) two groups $G:= (G,\id)$ and $H:=(H,\id)$ is
\[
\CohMotk ((H,\id),(G,\id)) = \big( \perm^\rfree_\kk(H,G) \big)^\natural ,
\]
the usual idempotent-completion of the additive Hom category in $\perm^\rfree_\kk$.
Its objects $(M,\mu)$ can be identified with the images $\mathrm{Im}(\mu)$ taken in the abelian category $\Bimod_\kk(H,G)$ of all $G,H$-bimodules, and the latter are summands of (right-free) permutation bimodules in the usual sense. (See \Cref{Exa:perm}.)

We have a pseudo-functor $\cohmot\colon \gpd^\op\to \CohMotk$ composed of the pseudo-functors~$\mot\colon\gpd^\op\to \Motk$ of~\eqref{eq:mot} and $\cat{P}\colon\Motk\onto \CohMotk$ of \Cref{Thm:2-Yoshida}. As with $\Motk$, for any finite groupoid~$G$, we still write $G$ for the object~$(G,\id)$ in~$\CohMotk$.

\Cref{Thm:UP-coh} will justify the motivic terminology.
\end{Def}

\begin{Def} \label{Def:coh-2Mack}
Recall the definition of a \emph{cohomological Mackey 2-functor} (\Cref{Def:co2mack-intro}).
Extending \Cref{Not:bimacks}, in the following we write
\[
\biCoMackk \subset \biMackk
\quad \textrm{ and } \quad
\bicCoMackk \subset \bicMackk
\]
for the 1- and 2-full sub-bicategories of those (idempotent-complete or not) Mackey 2-functors which are \emph{cohomological}.
\end{Def}

\begin{Thm}[Universal property]
\label{Thm:UP-coh}
The pseudo-functor $\cohmot\colon \gpd^\op\to \perm^\rfree_\kk$ induces by precomposition biequivalences of 2-categories
\[
\PsFun_\kk (\perm^\rfree_\kk, \ADD_\kk) \overset{\sim}{\to} \biCoMackk
\quad \textrm{and} \quad
\PsFun_\kk (\CohMotk, \ADDick) \overset{\sim}{\to} \bicCoMackk
\]
where $\PsFun_\kk$ denotes 2-categories of $\kk$-linear (hence additive) pseudo-functors, pseudo-natural transformations and modifications. In particular, every idempotent-complete $\kk$-linear cohomological  Mackey 2-functor factors uniquely up to isomorphism through $\CohMotk$ as claimed in \Cref{Thm:univ-cohom-2-Mackey-intro}.
\end{Thm}

\begin{proof}
By the universal property of Mackey 2-motives (see \Cref{Thm:UP-general}), the canonical embedding $\mot\colon \gpd^\op\to \kk\bisethatrf$ induces a biequivalence
\[
\PsFun_\kk (\kk\bisethatrf, \ADD_\kk) \overset{\sim}{\to} \biMack_\kk .
\]
This restricts to a biequivalence between, on the left, $\kk$-linear pseudo-functors $\widehat{\MM}$ annihilating the cohomological 2-cells~\eqref{eq:coh-rel},
and, on the right, (rectified) Mackey 2-functors $\MM$ which are cohomological. Combined with the (evident) biequivalences arising from the 2-universal property of the quotient $\kk\bisethatrf \twoheadrightarrow \kk\bisethatrf / \biker(\cat P)$ and from the biequivalence $\overline{\cat P}\colon \kk\bisethatrf / \biker(\cat P) \overset{\sim}{\to}\perm^\rfree_\kk$ of \Cref{Thm:2-Yoshida}, this yields the first claimed biequivalence.
Since $\ADDick$ is block-complete, the second claimed biequivalence follows readily from the first one by the universal property~\cite[Thm.\,A.7.23]{BalmerDellAmbrogio20} of the block-completion $\CohMotk  =  (\perm^\rfree_\kk)^\flat$.
\end{proof}

\begin{Cor} \label{Cor:univ}
For every cohomological Mackey 2-motive $X$, the composite
\[
\xymatrix{
\gpd^\op \ar[rr]^-{\cohmot} && \CohMotk \ar[rrr]^-{\CohMotk (X, -)} &&& \ADDick
}
\]
is a cohomological Mackey 2-functor. \textup(\Cf \cite[Ch.\,7.2]{BalmerDellAmbrogio20}.\textup)
\end{Cor}

\begin{proof}
Immediate from the second biequivalence in \Cref{Thm:UP-coh}.
\end{proof}

%------------------------------------------------------------------------------
%------------------------------------------------------------------------------
%
\section{Motivic decompositions}
\label{sec:motivic-decompositions}%
%\bigbreak
%------------------------------------------------------------------------------

In this section, we take a closer look at the bicategory of $\kk$-linear cohomological Mackey 2-motives $\CohMotk := (\perm^\rfree_\kk)^\flat$ of \Cref{Def:CoMot}, in order to compare cohomological and general motives.

Let us start with a few words about motivic decompositions.
Any decomposition of the Mackey 2-motive $\mot(G)\simeq X_1 \oplus \ldots \oplus X_n$ in $\Motk$ can be realized for \emph{any} Mackey 2-functor~$\MM$ as a decomposition of the additive category $\MM(G)$:
\[
\MM(G)\simeq \widehat{\MM}(X_1) \oplus \ldots \oplus \widehat{\MM}(X_n)
\]
where $\widehat{\MM}$ is the pseudo-functor extending~$\MM$ to Mackey 2-motives. See after \eqref{eq:mot-intro}.
Motivic decompositions are universal, in that they only depend on the group $G$ and not on the particular Mackey 2-functor~$\MM$.
See \cite[\S7.4-5]{BalmerDellAmbrogio20} for details.
\begin{Rem} \label{Rem:general-blocks}
In \cite[Ch.\,7.4]{BalmerDellAmbrogio20}, we found an explicit isomorphism of $\kk$-algebras between the motivic algebra of 2-cells $\End_{\kk\Spanhatrf}(\Id_G)$ of a group $G$ and the socalled \emph{crossed Burnside algebra} $\xBurk(G)$ first studied by Yoshida~\cite{Yoshida97}. Concretely, $\xBurk(G)$ is a finite free $\kk$-module generated by the set of $G$-conjugacy classes $[H,a]_G$ of pairs $(H,a)$ with $H\leq G$ a subgroup and $a\in \mathrm C_G(H)$ an element of the centralizer of~$H$, equipped with the (commutative!) multiplication induced by the formula
\[
[K,b]_G \cdot [H,a]_G = \sum_{[g] \in K\bs G/H} [K\cap {}^g H , bgag^{-1}]_G .
\]
Thus every general Mackey 2-motive is equivalent to a direct sum of pairs $(G,e)$ with $G$ a finite group and $e=e^2\in \xBurk(G)$ an idempotent.
\end{Rem}

An analogous discussion can be done for \emph{cohomological} motivic decompositions of~$\cohmot(G)$ in~$\CohMotk$. To understand the analogue of the crossed Burnside algebra in the cohomological setting, we need to understand $\End_{\perm^\rfree_\kk}(\Id_G)$.

\begin{Rem} \label{Rem:blocks}
For a finite group~$G$, the 2-cell endomorphism $\kk$-algebra in~$\perm^\rfree_\kk$ can be easily identified with the center of the group algebra~$\kk G$:
\[
\End_{\perm^\rfree_\kk}(\Id_G) \cong \mathrm Z(\kk G).
\]
Indeed, if $\varphi$ is an equivariant endomorphisms of the bimodule $\Id_G= {}_G\kk G_{G}$, the image~$\varphi(1_G)$ determines $\varphi$ and belongs to the center, since $g \varphi(1_G)= \varphi(g) = \varphi(1_G)g$ for all $g\in G$. Conversely, it is clear that any element of the center may serve as~$\varphi(1_G)$.
Hence, by the definition of the block-completion, every cohomological Mackey 2-motive is equivalent to a direct sum of pairs $(G,f)$ with $G$ a group and $f=f^2\in \kk G$ a central idempotent.
It is indecomposable if and only if $f$ is primitive.
\end{Rem}

\begin{Rem} \label{Rem:unique-decomp}
The decomposition of a Mackey 2-motive into a direct sum of indecomposable ones, both in $\Motk$ (as in \Cref{Rem:general-blocks}) and in $\CohMotk$ (as in \Cref{Rem:blocks}), is unique up to permutation and equivalence of the factors; see \cite[Cor.\,7.9]{DellAmbrogio22b}.
\end{Rem}

Following the above pattern, we now obtain:
\begin{Thm}
\label{Thm:universal-block-factorization}
If $\MM$ is any cohomological Mackey 2-functor, its value category at each finite group $G$ admits a canonical decomposition into direct factors
\[
\MM(G) \cong \bigoplus_b \MM(G;b)
\]
indexed by the blocks (primitive central idempotents) $b$ of the group algebra~$\kk G$.
\end{Thm}

\begin{proof}
Just apply the reasoning of \cite[\S\,7.5]{BalmerDellAmbrogio20} to the motivic decompositions of \Cref{Rem:blocks}, as explained above.
\end{proof}

Recall the linearization pseudo-functor $\cat P\colon \kk \Spanhatrf \to \perm^\rfree_\kk$ of \Cref{sec:cohom-2-motives}.
By applying block-completion $(-)^\flat$ to both sides, it extends to a pseudo-functor
\[
\cat P\colon \Motk \longrightarrow \CohMotk
\]
comparing general and cohomological Mackey 2-motives. The following result gives a very concrete description of the effect of~$\cat{P}$ on equivalence classes of motives:

\begin{Thm} \label{Thm:motivic-comparison}
For every finite group~$G$, there is a well-defined surjective morphism of commutative rings
$
\rho_G \colon \xBurk(G) \to \mathrm Z(\kk G)
$
sending a basis element $[H,a]_H$ to $\sum_{[x]\in G/H} xax^{-1}= \sum_{[y]\in H\bs G}y^{-1}xy$.
Collectively, they govern the behavior of $\cat{P}$ on equivalence classes of 2-motives, meaning that $\cat{P}$ maps the general Mackey 2-motive $\oplus_i (G_i,e_i)$ (see \Cref{Rem:general-blocks}) to the cohomological Mackey 2-motive $\oplus_i (G_i,\rho_{G_i}(e_i))$ (see \Cref{Rem:blocks}), where of course $(G,0)\cong 0$ in both bicategories.
\end{Thm}

\begin{proof}
Define $\rho_G$ by the following diagram:
\[
\xymatrix{
\xBurk (G) \ar[d]_-{\textrm{\cite[Thm.\,7.4.5]{BalmerDellAmbrogio20}}}^-{\cong} \ar@{-->}[rr]^-{\rho_G} && Z(\kk G)  \\
  \End_{\kk\Spanhatrf}(\Id_G) \ar@{->>}[rr]^-{\cat{P}} &&  \End_{\CohMotk} (\Id_G) \ar[u]^-{\cong}_-{\textrm{\Cref{Rem:blocks}}}
}
\]
A direct inspection of the definitions reveals that $\rho_G$ is indeed given by the claimed formula.
(We use here the notation $\rho_G$ because this map is essentially a special case of the homonymous one studied in \cite[Ch.\,7.5]{BalmerDellAmbrogio20}; indeed $\rho_G([H,a]_G)$ corresponds in $\End_{\CohMotk} (\Id_G)$ to the composite equivariant map
\[
\xymatrix{
{}_G \kk G_G \ar@{=>}[r]^-{\reta} &
 {}_G \kk G \otimes_H  \kk G_G  \ar@{=>}[r]^-{\Ind (\gamma_a )} &
  {}_G \kk G \otimes_H  \kk G_G \ar@{=>}[r]^-{\leps} &
   {}_G \kk G_G
}
\]
with $\reta$ and $\leps$ as in \Cref{Rem:P-adj-groups} and $\gamma_a\colon {}_H\kk G_G\Rightarrow {}_H\kk G_G$  defined by $g\mapsto ag$.)
Each $\rho_G$ is a surjective map of commutative rings because it is a composite of such; in particular, $\cat{P}$ is 2-full by \Cref{Thm:2-Yoshida}. The remaining claims are immediate from \Cref{Rem:general-blocks} and \Cref{Rem:blocks}.
\end{proof}

\begin{Rem} \label{Rem:duhduh!}
Note that $\rho_G([H,1]_G)= [G\!:\!H]\in \kk \subset \kk G$ (of course!).
Recall that the ordinary \emph{Burnside ring} $\Burk(G)$ identifies with the subalgebra of $\xBurk(G)$ generated by such basis elements, hence is sent to $\kk$ by~$\rho_G$.
It follows that, on \emph{cohomological} Mackey 2-functors, the idempotents of $\Burk(G)$ do not produce any interesting factor.
This is wrong for \emph{non}-cohomological Mackey 2-functors, \eg for equivariant stable homotopy theory (see \cite[Ex.\,4.3.8]{BalmerDellAmbrogio20} and \cite[App.\,A]{GreenleesMay95}).
\end{Rem}

Very recently, Oda, Takegahara and Yoshida \cite{OTY22} have proved explicit refinements of \Cref{Thm:motivic-comparison} for some important cases of the base ring~$\kk$, by describing the primitive idempotents of $\xBurk(G)$ and their behavior under~$\rho_G$ in terms of the structure of~$G$. Perhaps not surprisingly, this problem appears to be subtle even for $\kk$ a well-behaved local ring, where there are connections with character theory; \cf \cite{Bouc03}.
On our part, we can offer the following general lifting result:

\begin{Cor} \label{Cor:ess-surj}
Assume the commutative ring $\kk$ is a complete local noetherian ring, for instance, a field.
In this case, the pseudo-functor
$
\cat{P}\colon \Motk \to \CohMotk
$
is essentially surjective on objects and 1-cells. In particular, it is genuinely a quotient pseudo-functor of $\kk$-linear bicategories -- not just up to retracts.
\end{Cor}

\begin{proof}
By \Cref{Def:CoMot} and \Cref{Thm:2-Yoshida}, we already know that each object or 1-cell of the block-completion $\CohMotk= (\perm^\rfree_\kk)^\flat$ is a retract of an object or 1-cell in the image of~$\cat{P}$.
Hence it suffices to show that arbitrary idempotents can be lifted along the algebra morphisms
$
\cat{P}\colon \End_{\kk \Spanhatrf}(S)  \to \End_{\perm^\rfree_\kk}(\cat{P}(S))
$
for all 1-cells~$S$ of~$\kk \Spanhatrf \cong \kk \bisethatrf$.
By construction, the latter are surjective (by \Cref{Thm:2-Yoshida}) morphisms of (noncommutative) finite dimensional $\kk$-algebras. For $\kk$ as above, the result now follows from the general lifting theorem \cite[Thm.\,4.7.1]{Linckelmann18a}. 
\end{proof}

%------------------------------------------------------------------------------
%    Bibliography
%------------------------------------------------------------------------------

% uncomment the latter (and comment out the references below) to recompile the bibliography from articles.bib
%\bibliographystyle{alpha}
%\bibliography{articles}

\begin{thebibliography}{OTY22}

\bibitem[Bal15]{Balmer15}
Paul Balmer.
\newblock Stacks of group representations.
\newblock {\em J. Eur. Math. Soc. (JEMS)}, 17(1):189--228, 2015.

\bibitem[BD20]{BalmerDellAmbrogio20}
Paul {Balmer} and Ivo {Dell'Ambrogio}.
\newblock {\em {Mackey 2-functors and Mackey 2-motives}}.
\newblock European Mathematical Society (EMS), Z\"urich, 2020.

\bibitem[BD21]{BalmerDellAmbrogio21}
Paul Balmer and Ivo Dell'Ambrogio.
\newblock Green equivalences in equivariant mathematics.
\newblock {\em Math. Ann.}, 379(3-4):1315--1342, 2021.

\bibitem[B{\'{e}}n00]{Benabou00}
Jean B{\'{e}}nabou.
\newblock {D}istributors at work.
\newblock Course notes
  \url{http://www.entretemps.asso.fr/maths/Distributors.pdf}, 28 pages, 2000.

\bibitem[Bou03]{Bouc03}
Serge Bouc.
\newblock The {$p$}-blocks of the {M}ackey algebra.
\newblock {\em Algebr. Represent. Theory}, 6(5):515--543, 2003.

\bibitem[Bou10]{Bouc10}
Serge Bouc.
\newblock {\em Biset functors for finite groups}, volume 1990 of {\em Lecture
  Notes in Mathematics}.
\newblock Springer-Verlag, Berlin, 2010.

\bibitem[Del21]{DellAmbrogio21pp}
Ivo Dell'Ambrogio.
\newblock Green 2-functors.
\newblock {\em Trans. Amer. Math. Soc.}, to appear, preprint 2021.

\bibitem[Del22a]{DellAmbrogio22a}
Ivo Dell'Ambrogio.
\newblock Axiomatic representation theory of finite groups by way of groupoids.
\newblock In {\em Equivariant topology and derived algebra}, volume 474 of {\em
  London Math. Soc. Lecture Note Ser.}, pages 39--99. Cambridge Univ. Press,
  Cambridge, 2022.

\bibitem[Del22b]{DellAmbrogio22b}
Ivo Dell'Ambrogio.
\newblock On {K}rull-{S}chmidt bicategories.
\newblock {\em Theory Appl. Categ.}, 38:Paper No. 8, 232--256, 2022.

\bibitem[DH21]{DellAmbrogioHuglo21}
Ivo Dell'Ambrogio and James Huglo.
\newblock On the comparison of spans and bisets.
\newblock {\em Cahiers Topologie G\'{e}om. Diff\'{e}rentielle Cat\'{e}g.},
  62(1):63--104, 2021.

\bibitem[GM95]{GreenleesMay95}
J.~P.~C. Greenlees and J.~P. May.
\newblock Generalized {T}ate cohomology.
\newblock {\em Mem. Amer. Math. Soc.}, 113(543):viii+178, 1995.

\bibitem[Hug19]{Huglo19pp}
James Huglo.
\newblock {P}h{D} thesis, {U}niversit{\'e} de {L}ille~1.
\newblock Available at
  \url{https://math.univ-lille1.fr/~dellambr/TheseVersionFinaleHUGLO.pdf},
  2019.

\bibitem[Lin18]{Linckelmann18a}
Markus Linckelmann.
\newblock {\em The block theory of finite group algebras. {V}ol. {II}},
  volume~92 of {\em London Mathematical Society Student Texts}.
\newblock Cambridge University Press, 2018.

\bibitem[Mai22]{Maillard22}
Jun Maillard.
\newblock A categorification of the {C}artan-{E}ilenberg formula.
\newblock {\em Adv. Math.}, 396(33): 1081--87, 2022.

\bibitem[Mil17]{Miller17}
Haynes Miller.
\newblock The {B}urnside bicategory of groupoids.
\newblock {\em Bol. Soc. Mat. Mex. (3)}, 23(1):173--194, 2017.

\bibitem[OTY22]{OTY22}
Fumihito Oda, Yugen Tagegahara, and Tomoyuki Yoshida.
\newblock Crossed {B}urnside rings and cohomological {M}ackey 2-motives.
\newblock Preprint \url{https://arxiv.org/abs/2201.04744}, 2022.

\bibitem[Web93]{Webb93}
Peter Webb.
\newblock Two classifications of simple {M}ackey functors with applications to
  group cohomology and the decomposition of classifying spaces.
\newblock {\em J. Pure Appl. Algebra}, 88(1-3):265--304, 1993.

\bibitem[Web00]{Webb00}
Peter Webb.
\newblock A guide to {M}ackey functors.
\newblock In {\em Handbook of algebra, {V}ol. 2}, pages 805--836.
  North-Holland, Amsterdam, 2000.

\bibitem[Yos83]{Yoshida83b}
Tomoyuki Yoshida.
\newblock On {$G$}-functors. {II}. {H}ecke operators and {$G$}-functors.
\newblock {\em J. Math. Soc. Japan}, 35(1):179--190, 1983.

\bibitem[Yos97]{Yoshida97}
Tomoyuki Yoshida.
\newblock Crossed {$G$}-sets and crossed {B}urnside rings.
\newblock {\em S\={u}rikaisekikenky\={u}sho K\={o}ky\={u}roku}, 991:1--15,
  1997.
\newblock Group theory and combinatorial mathematics (Japanese) (Kyoto, 1996).

\end{thebibliography}

%------------------------------------------------------------------------------
\end{document}